\documentclass[12pt, twoside, english, headsepline]{article}

\usepackage[margin=30mm]{geometry}
\usepackage{tikz,graphicx}
\usetikzlibrary{arrows}
\usetikzlibrary{patterns}
\newcommand{\midarrow}{\tikz \draw[-triangle 90] (0,0) -- +(.1,0);}
\tikzset{
  mid arrow/.style={postaction={decorate,decoration={
        markings,
        mark=at position .5 with {\arrow[#1]{stealth}}
      }}},
}
\usetikzlibrary{decorations.markings} 
\tikzset
 {every pin/.style = {pin edge = {<-}}, 
  > = stealth, 
  flow/.style = 
   {decoration = {markings, mark=at position #1 with {\arrow{>}}},
    postaction = {decorate}
   },
  flow/.default = 0.5,   
  main/.style = {line width=1pt}
 }

\usepackage{amsmath,amsthm,amssymb,amsfonts, mathtools, amsbsy, dsfont}
\usepackage[utf8]{inputenx}
\usepackage{xcolor}
\usepackage{microtype}
\usepackage{hyperref}
\hypersetup{linkcolor=blue}

\newcommand*\dif{\mathop{}\!\mathrm{d}}  
\newcommand{\Conv}{\mathop{\scalebox{2.8}{\raisebox{-0.2ex}{$\ast$}}}}

\allowdisplaybreaks 

\newtheorem{theorem}{Theorem}[section] 
\newtheorem{corollary} {Corollary}[section] 
\newtheorem{proposition}{Proposition}[section] 

\theoremstyle{definition} 
 
\newtheorem{example}{Example}[section]
\newtheorem{remark}{Remark}[section] 
\numberwithin{equation}{section} 

\providecommand{\keywords}[1]
{
  \small	
  \textbf{\textit{Keywords: }} #1
}
\providecommand{\MSC}[1]
{
  \small	
  \textit{2020 MSC: } #1   
}

\title{General Airy-type equations, heat-type equations and pseudo-processes}
\author{Fabrizio Cinque$^1$ and Enzo Orsingher$^2$\\
        \small Department of Statistical Sciences, Sapienza University of Rome, Italy \\
        \small $^1$fabrizio.cinque@uniroma1.it $^2$enzo.orsingher@uniroma1.it
}

\begin{document}

\maketitle

\begin{abstract}
We present a systematic study of higher-order Airy-type differential equations providing the explicit form of the solutions, deriving their power series expansions and a probabilistic interpretation. Under suitable convergence hypotheses, we compute their integral on the real line and, by means of complex integration, we provide alternative explicit forms. We then focus on the differential equations governing their derivatives, their products, their convolutions and higher-order Scorer type equations. Then, we study higher-order heat-type fractional Cauchy, showing that their fundamental solutions can be expressed in terms of Airy-type functions and their convolutions, recovering as special cases several results of appeared in previous papers. Furthermore, pseudo-processes theory permits us to give nice interpretations of the results, extending them in the case of equations involving different fractional operators and study the moments of the solutions.
\end{abstract} \hspace{10pt}


\MSC{Primary 33C10, 34A05, 35K25; Secondary 60K99}

\section{Introduction}

The physical origin of Airy functions is the mathematical description of waves propagating through a hole. Airy \cite{A1838} showed that the intensity of light on the line (caustic) separating the shaded zone from the light one for a monocromatic light at point $P$ with distance $R$ from the hole can be approximated, by means of geometric assumptions, with the function of kind
\begin{equation}\label{formulaApprossimazioneAiry}
 u_P \sim \int_I e^{-ik\bigl(aw + b w^3\bigr)} \dif w,
\end{equation}
with suitable parameters $a,b, I$ (see Vall\'{e}e-Soares \cite{VS2010} page 115).
Inspired by the Airy integral, in this paper we study suitable generalization of this function by solving the higher-order Airy-type ordinary differential equations and by increasing the order of approximation given in (\ref{formulaApprossimazioneAiry}), i.e. by considering functions of the form
\begin{equation}\label{AiryConvoluzioneIntro}
 u(x) = \int_0^\infty e^{-ixw + \sum_{k=0}^N a_k \frac{w^{2k+1}}{2k+1} } \dif w.
\end{equation}
We prove that function (\ref{AiryConvoluzioneIntro}) can be expressed as the convolution of higher-order Airy functions (also known as hyper-Airy functions) and can be related to suitable pseudo-processes.

Hyper-Airy functions recently appeared in several papers such as \cite{A2017, AA2014, AA2020, AA2024, MO2023}. In particular, hyper-Airy functions emerge in the study of higher-order heat-type equations of real order such as, with $\alpha>1$,
\begin{equation}\label{problemaParzialeCauchyAiryRealeIntro}
\begin{cases}
\frac{\partial u}{\partial t}(t,x) = D^\alpha_x u(t,x), \ \ \ t\ge0,\ x\in \mathbb{R},  \\
u(0,x) = f(x), \ \ \forall\ x,\\
\lim_{x\rightarrow \pm \infty} u(t,x) = 0, \ \ \forall\ t,
\end{cases}
\end{equation}
where $D^\alpha_x$ is the Riesz space fractional operator and it is implicitly defined by means of the Fourier transform, that is, for $\gamma\in\mathbb{R}$,
\begin{equation}\label{trasformataFourierOperatoreRiesz}
\mathcal{F} \bigl(D^\alpha_x f \bigr) (\gamma) = 
\begin{cases}
\begin{array}{l l}
 -i\,\text{sgn}(\gamma) |\gamma|^\alpha \mathcal{F}f(\gamma), & \alpha\not\in\mathbb{N}\\
(-i\gamma)^\alpha \mathcal{F}f(\gamma), & \alpha\in\mathbb{N},
\end{array}
\end{cases}
\end{equation}
meaning that for integer values of $\alpha$ the Riesz operator coincides with the classical derivative.

In the paper \cite{MO2023}, the authors show that the solution to the Cauchy problem (\ref{problemaParzialeCauchyAiryRealeIntro}) with $u(0,x) = f(x) = \delta(x)$, the Dirac delta function centered in $0$, reads 
\begin{equation}\label{soluzioneProblemDifferenzialeParzialeIperAiryReale}
u_\alpha(t, x) = \frac{1}{\pi}\int_0^\infty \cos(xw + tw^\alpha) \dif w = \frac{1}{\sqrt[\alpha]{\alpha t}} A_{\alpha}\Bigl(\frac{x}{\sqrt[\alpha]{\alpha t}}\Bigr),
\end{equation}
where 
\begin{equation}\label{funzioneAiryReale}
A_{\alpha}(x) = \frac{1}{\pi} \int_0^\infty \cos\Bigl(xw +\frac{w^\alpha}{\alpha}\Bigr) \dif w, \ \ x\in\mathbb{C}, \ \alpha > 1,
\end{equation}
is the hyper-Airy function of order $\alpha$ (in the sense given in the article \cite{MO2023}).
If $\alpha = 3$, formula (\ref{funzioneAiryReale}) yields the classical Airy function.

In the paper \cite{MO2023} the authors proved that the solution $u_\alpha$ can be written as the expectation of a suitable function of a random variable, such as
\begin{equation}\label{funzioneAiryAspettativa}
u_{\alpha}(t,x) = \frac{1}{\pi x}\mathbb{E} \biggl[e^{-b_\alpha xG_{\alpha}(1/t)} \sin\Bigl(a_\alpha x G_\alpha(1/t)\Bigr) \biggr],
\end{equation}
with $G $ having probability density function $f_G(w;\alpha, \tau) = \frac{\alpha w^{\alpha-1}}{\tau} e^{-w^\alpha/\tau}$,  
$a_\alpha = \cos\bigl(\pi/(2\alpha)\bigr)$ and $b_\alpha = \sin\bigl(\pi/(2\alpha)\bigr)$.
This result was already presented for the odd case ($\alpha=2n+1$) in \cite{ODo2012}. Below, we provide a similar representation for the functions we deal with.
\\
%

In the case of $\alpha = 2n+1$, the starting point of the research on Airy-type functions, in analogy to the well-known third-order Airy functions, can be found in the $2n$-th-order ordinary differential equation 
\begin{equation*}
y^{(2n)} + (-1)^n xy = 0, \ \  x \in \mathbb{C},\ n\in\mathbb{N},
\end{equation*}
 which in the case of $n=1$ admits the classical Airy functions as solutions, which are, for $x\in\mathbb{R}$,
\begin{equation}\label{funzioniAiryClassiche}
\text{Ai}(x) = \frac{1}{\pi} \int_0^\infty \cos\Bigl(xw +\frac{w^3}{3}\Bigr) \dif w \ \ \text{and}\ \ \text{Bi}(x) = \frac{1}{\pi}\int_0^\infty \biggl[e^{xw -\frac{w^3}{3}}+ \sin\Bigl(xw +\frac{w^3}{3}\Bigr)\biggr] \dif w.
\end{equation}

In this work, we study the more general form 
\begin{equation}\label{equazioneDifferenzialeGeneraleIntro}
y^{(n-1)} + cxy = 0, \ \ \ x\in \mathbb{C},\ c\in \mathbb{R}\setminus{\{0\}}, \ \text{integer } n\ge 2,
\end{equation}
which has been studied in some cases, see \cite{A2017, GHPD2013}.
By means of a spectral approach we derive the $n-1$ independent solutions to (\ref{equazioneDifferenzialeGeneraleIntro}), that, for example, for $n$ odd and real $c\not=0$ or $n$ even and $c<0$, can be listed as follows, with $x\in \mathbb{C}$ and natural $k \le n/2$,
\begin{align}
y_k^+ (x) &= \int_0^\infty e^{-(-c)^{n-1}\frac{w^n}{n}}\Biggl[e^{-cwx} - e^{-cwx a_k}\cos\Bigl( cwx b_k - \frac{2k\pi}{n} \Bigr) \Biggr]  \dif w, \label{soluzioneAiryGeneralek+}\\
&= \sum_{j=0}^\infty \frac{(-\text{sgn}(c)x)^j}{j!} \Biggl[1-\cos\Bigl((j+1)\frac{2k\pi}{n} \Bigr) \Biggr]\Gamma\Bigl(\frac{j+1}{n}\Bigr)(n|c|)^{\frac{j+1}{n}-1},\nonumber
\end{align}
\begin{align}
y_k^-(x) &= - \int_0^\infty e^{-(-c)^{n-1}\frac{w^n}{n}- cwx a_k} \sin\Bigl( cwx b_k - \frac{2k\pi}{n} \Bigr) \dif w \label{soluzioneAiryGeneralek-}\\
&= \sum_{j=0}^\infty \frac{(-\text{sgn}(c)x)^j}{j!} \sin\Bigl((j+1)\frac{2k\pi}{n} \Bigr) \Gamma\Bigl(\frac{j+1}{n}\Bigr)(n|c|)^{\frac{j+1}{n}-1}, \nonumber
\end{align}
where $a_k = \cos(2k\pi/n)$ and $b_k = \sin(2k\pi/n)$.
Solutions $y_k^+$ and $y_k^-$ differ because of the presence of a different trigonometric function and an additional term in $y_k^+$. This recalls the structure of classical Airy functions $\text{Ai}$ and $\text{Bi}$, displayed in (\ref{funzioniAiryClassiche}), which can be seen as a particular case of (\ref{soluzioneAiryGeneralek+}) and (\ref{soluzioneAiryGeneralek-}). More generally, thanks to the integration in the complex plane $\mathbb{C}$, we prove that these solutions reduce to the known hyper-Airy functions for $n$ odd and some specific $k$.

Functions (\ref{soluzioneAiryGeneralek+}) and (\ref{soluzioneAiryGeneralek-}) are analytic and they admit the given power series representation, which is derived as the Taylor expansion in $x=0$. In some cases, the solutions to (\ref{equazioneDifferenzialeGeneraleIntro}) have finite integral on the real line (or on either the positive or negative semi-axis); we provide some examples below.
\\Furthemore, we show that the two terms of (\ref{soluzioneAiryGeneralek+}) satisfy a higher-order Scorer-type equation, that is,
$$ y^{(n-1)} - cxy = 1, \ \ \ x\in\mathbb{C}.$$

In accordance with the probabilistic interpretation (\ref{funzioneAiryAspettativa}) given for the hyper-Airy functions, we show that the solutions of the general differential equation (\ref{equazioneDifferenzialeGeneraleIntro}) admit similar representations; for instance, we can write function (\ref{soluzioneAiryGeneralek-}) as
$$ y_k^-(x) = - \frac{\Gamma(1/n)}{(-nc)^{1-\frac{1}{n}}}\mathbb{E}\Biggl[ e^{-c\sqrt[n]{Y}xa_k} \sin\Bigl( c\sqrt[n]{Y}xb_k - \frac{2k\pi}{n}\Bigr) \Biggr], $$
with $Y\sim Gamma\Bigl(\frac{1}{n}, \frac{(-c)^{n-1}}{n}\Bigr)$ being a gamma random variable.
\\

From the structure of the differential equation (\ref{equazioneDifferenzialeGeneraleIntro}) we derive the non-homogeneous ordinary differential equation satisfied by the derivatives of the solutions to (\ref{equazioneDifferenzialeGeneraleIntro}) as well as the higher-order equations governing their product. We conclude Section 2 investigating the ordinary differential equation solved by the convolution (and its derivatives) of the solutions to equations of type (\ref{equazioneDifferenzialeGeneraleIntro}). By assuming $f_{n}$ solving (\ref{equazioneDifferenzialeGeneraleIntro}) with $c = c_n$, function $y = \Conv_{k=1}^N f_{n_k}$ satisfies
\begin{equation}
\sum_{k=1}^N y^{(n_k - 1)} \prod_{\substack{h=1\\ h\not=k}}^N c_{n_h} + x y \prod_{k=1}^N c_{n_k} = 0.
\end{equation}

In Section 3 we study a generalization of problem (\ref{problemaParzialeCauchyAiryRealeIntro}) with more general differential equations, with real $\alpha_1,\dots,\alpha_N\not=0$ and suitable real $a_1,\dots a_N$,
\begin{equation}\label{equazioneCaloreGeneraleIntro}
 \frac{\partial u}{\partial t}(t,x) = \sum_{k=1}^N a_k D_x^{\alpha_k} u(t,x),\ \ \ t\ge0,\ x\in\mathbb{R}, 
\end{equation}
showing that the solution is of the type of function (\ref{AiryConvoluzioneIntro}) and that it can be expressed as the convolution of the hyper-Airy functions (\ref{funzioneAiryReale}). We also state in detail the relationship between the higher-order heat-type equations and the Airy-type ones, proving that the solution to (\ref{equazioneCaloreGeneraleIntro}) for $N=1$ and integer order $n$ satisfies equation (\ref{equazioneDifferenzialeGeneraleIntro}) with $c=1/(nta_1)$.

Finally, we study the pseudo-processes related to problem (\ref{equazioneCaloreGeneraleIntro}) and those involving a time-fractional operator as well as the Riesz-Feller (space) fractional operator, showing that the solution can still be interpreted as the pseudo-density of the position of a suitably time-changed pseudo-process. In the integer case, we also compute the moments of the pseudo-processes.

\section{Generalized ordinary Airy-type differential equations}\label{sezioneMotoPianoTreDirezioni}

Inspired by the differential equation governing the hyper-Airy functions,
\begin{equation}\label{equazioneAiryDispari}
y^{(2n)} + (-1)^n xy = 0, \ \ \ x\in \mathbb{C},\ c\in \mathbb{R}\setminus{\{0\}}, \ n\in\mathbb{N},
\end{equation}
we approach a more general differential equation stating the following theorem.

\begin{theorem}\label{teoremaSoluzioneEquazioneDifferenzialeGenerale}
Let $n\ge2$ be an integer number. The ordinary differential equation
\begin{equation}\label{equazioneDifferenzialeGenerale}
y^{(n-1)} + cxy = 0, \ \ \ x\in \mathbb{C},\ c\in \mathbb{R}\setminus{\{0\}},
\end{equation}
admits the following $n-1$ linearly independent solutions:
\begin{itemize}
\item if $n$ odd or $n$ even and $c<0$,
\begin{align}
&y_k^+ (x) = \int_0^\infty e^{-(-c)^{n-1}\frac{w^n}{n}} \Biggl[e^{-cwx} - e^{-cwx a_k}\cos\Bigl( cwx b_k - \frac{2k\pi}{n} \Bigr) \Biggr] \dif w, \label{soluzioneEquazioneDifferenzialeGeneraleSomma} \\
&y_k^-(x) = - \int_0^\infty e^{-(-c)^{n-1}\frac{w^n}{n}- cwx a_k} \sin\Bigl( cwx b_k - \frac{2k\pi}{n} \Bigr) \dif w, \label{soluzioneEquazioneDifferenzialeGeneraleDifferenza}
\end{align}
with natural $k \le n/2$ and $a_k = \cos(2k\pi/n)$ and $b_k = \sin(2k\pi/n)$;
\item if $n$ even and $c>0$,
\begin{align}
&y_k^+ (x) = \int_0^\infty e^{-c^{n-1}\frac{w^n}{n}}\dif w  \nonumber\\ 
&\ \ \ \times \Biggl[e^{i\frac{\pi}{n}}e^{-cwxe^{i\pi/n}} - e^{-cwx a_k}\cos\Bigl( cwx b_k - \frac{2k+1}{n}\pi \Bigr) \Biggr] , \ \ \ k=1,\dots,\frac{n}{2}-1, \label{soluzioneEquazioneDifferenzialeGeneraleSommaPari}  \\
&y_k^-(x) = - \int_0^\infty e^{-c^{n-1}\frac{w^n}{n}- cwx a_k} \sin\Bigl( cwx b_k - \frac{2k+1}{n}\pi \Bigr) \dif w, \ \ \  k =0,\dots, \frac{n}{2}-1, \label{soluzioneEquazioneDifferenzialeGeneraleDifferenzaPari}
\end{align}
with $a_k = \cos\bigl((2k+1)\pi/n\bigr)$ and $b_k = \sin\bigl((2k+1)\pi/n\bigr)$.
\end{itemize}
\end{theorem}

We note that in the case of $n =2$, as well known, the solution of (\ref{equazioneDifferenzialeGenerale}) is the Gaussian function. For instance, in the case of $c <0$, by means of (\ref{soluzioneEquazioneDifferenzialeGeneraleSomma}) with $k=1$ we readily obtain
\begin{align*}
y_1^+(x)& = \int_0^\infty e^{-(-c)\frac{w^2}{2}}\Bigl[ e^{-cwx} + e^{cwx}\Bigr] \dif w\\
& =  2\sum_{k\ge0} \frac{(cx)^{2k}}{(2k)!}\int_0^\infty e^{-(-c)\frac{w^2}{2}} w^{2k} \dif w\\
& =  \sum_{k\ge0} \frac{(cx)^{2k}}{(2k)!}\int_0^\infty e^{-(-c)\frac{y}{2}} y^{k+\frac{1}{2}} \dif y\\
& =  \sum_{k\ge0} \frac{(cx)^{2k}}{(2k)!} \Gamma\biggl(k+\frac{1}{2}\biggr)\biggl(\frac{2}{-c}\biggr)^{k+\frac{1}{2}}\\
& =  \sum_{k\ge0} \biggl(\frac{-cx^2}{2}\biggr)^{k}	\frac{1}{k!} \sqrt{\frac{2}{-c}} \\
&= \sqrt{\frac{2}{-c}} e^{-\frac{c}{2}x^2}, \ \ \ c<0, \ x\in\mathbb{C},
\end{align*}
where we used the Gamma duplication formula, $\Gamma(k+1/2) = \Gamma(2k)2^{1-2k}/\Gamma(k)$. The interested reader can prove the case for $c>0$ in a similar manner by using (\ref{soluzioneEquazioneDifferenzialeGeneraleDifferenzaPari}) with $k=0$.

\begin{proof}
We use a spectral approach (also known as Laplace contour integrals), therefore we study a solution of the type $g(x) = \int_\mathcal{C} e^{-cxw}f(x)\dif w$ where $\mathcal{C}$ denotes a suitable contour in the complex plane $\mathbb{C}$ and $f:\mathbb{C}\longrightarrow\mathbb{C}$ is a suitable function. Now, we determine $f$ and $\mathcal{C}$.

By inserting $g$ into the equation (\ref{equazioneDifferenzialeGenerale}), by means of some calculation also involving an integration by parts, we obtain the equation
\begin{equation}\label{passaggioSoluzioneEquazioneDifferenzialeGenerale}
 \int_C e^{-cxw}\Bigl((-c)^{n-1}w^{n-1}f(w)+ f'(w)\Bigr) \dif x- f(w)e^{-cwx}\Big|_\mathcal{C} =0.
\end{equation}
Then, we derive $f$ by requiring $(-c)^{n-1}w^{n-1}f(w)+ f'(w)=0$ which leads to $f(w) = e^{-(-c)^{n-1}w^n/n}, w\in\mathbb{C}.$

Now, given the form of $f$ and equation (\ref{passaggioSoluzioneEquazioneDifferenzialeGenerale}), we derive the contour $\mathcal{C}$ by solving the equation $e^{-cwx-(-c)^{n-1}w^n/n}\Big|_\mathcal{C} =0$. First, we study this condition on the half-line coming out of the origin $\gamma_\theta$, that is expressed as $\phi_\theta(t) = te^{i\theta},\ t\in (0,\infty)$, and we calculate the angles $\theta\in[0,2\pi]$ which permit us to have a finite result:
\begin{align}
e^{-cwx-(-c)^{n-1}w^n/n}\Big|_{\gamma_\theta} &= \exp\Bigl(-(-c)^{n-1}\frac{t^n e^{in\theta}}{n} - cxte^{i\theta}\Bigr)\Big|_{t = 0}^{t \rightarrow \infty}\nonumber\\
&= -1 + \lim_{t\rightarrow \infty} \frac{\exp\Bigl(-i(-c)^{n-1}\sin(n\theta)t^n/n - ic \sin(\theta)x t\Bigr)}{\exp\Bigl((-c)^{n-1}\cos(n\theta)t^n/n + c\cos(\theta)xt \Bigr)}. \label{limiteSulContornoSemiretta}
\end{align}
Thus, we need 
\begin{equation}\label{condizioneConvergenzaLimiteContornoSemiretta}
\theta \text{ such that } \begin{array}{l} (i)\ \  \ (-c)^{n-1}\cos(n\theta) >0,\\
(ii) \ \ (-c)^{n-1}\cos(n\theta) =0 \text{ and } c\bigl(\cos(\theta)\Re(x)-\sin(\theta)\Im(x)\bigr) > 0 .
\end{array}
\end{equation}
From $(i)$ we obtain that the limit in (\ref{limiteSulContornoSemiretta}) is finite (and equal to $0$) for $\theta$ such that
\begin{equation}\label{finePrimaCondizioneConvergenzaLimiteContornoSemiretta}
\begin{cases}
\cos(n\theta) < 0, \ \ \ \text{if } c>0 \text{ and } n \text{ even},\\
\cos(n\theta)>0, \ \ \ \text{if } n \text{ odd or }c<0 \text{ and } n \text{ even}.
\end{cases}
\end{equation}
This condition is sufficient for constructing the solutions we need. However, the convergence result required by condition $(ii)$ of (\ref{condizioneConvergenzaLimiteContornoSemiretta}) will be useful in Remark \ref{remarkFormaSoluzioniTipoAiry} for some manipulation of the form of the solutions by means of complex integration (more precisely the application of Cauchy-Goursat theorem).

Now, we give the complete proof in the second case of (\ref{finePrimaCondizioneConvergenzaLimiteContornoSemiretta}), so we require $n$ odd or $n$ even and $c<0$ (note that these conditions imply that $(-c)^{n-1}=|c|^{n-1}\ge0$). In this case, the limit in (\ref{limiteSulContornoSemiretta}) is finite (and equal to $-1$) if $\cos(n\theta)>0$ (it is a sufficient condition because we are not considering the possibility in $(ii)$ of (\ref{condizioneConvergenzaLimiteContornoSemiretta})), that is if
\begin{equation}\label{insiemeTheta}
\theta\in \bigcup_{k=0}^{n-1} \Bigl(\frac{4k-1}{2n}\pi, \frac{4k+1}{2n}\pi\Bigr) =: I,
\end{equation}
which is a collection of intervals of size $\pi/n$.

\begin{figure}
\begin{minipage}{0.49\textwidth}
		\centering
\begin{tikzpicture}[scale = 0.77]

\draw[pattern=horizontal lines, pattern color=lightgray] (0,0) -- (4.6*0.95, 4.6*0.31) -- (4.6*0.95, 4.6*-0.31) -- cycle;
\draw[pattern=horizontal lines, pattern color=lightgray] (0,0) -- (4.6*0.59, 4.6*0.81) -- (0, 4.6) -- cycle;
\draw[pattern=horizontal lines, pattern color=lightgray] (0,0) -- (4.6*-0.59, 4.6*0.81) -- (4.6*-0.95, 4.6*0.31) -- cycle;
\draw[pattern=horizontal lines, pattern color=lightgray] (0,0) -- (4.6*-0.95, 4.6*-0.31) -- (4.6*-0.59, 4.6*-0.81) -- cycle;
\draw[pattern=horizontal lines, pattern color=lightgray] (0,0) -- (0, -4.6) -- (4.6*0.59, 4.6*-0.81) -- cycle;

		\draw[->, thick, gray] (-5,0) -- (5,0) node[below, scale = 1, black]{$\pmb{\Re}$};
		\draw[->, thick, gray] (0,-5) -- (0,5) node[left, scale = 1, black]{$\pmb{\Im}$};
\draw (0,0) -- (4.5*0.95, 4.5*0.31);
\draw (0,0) -- (4.5*0.59, 4.5*0.81);
\draw (0,0) -- (4.5*-0.59, 4.5*0.81);
\draw (0,0) -- (4.5*-0.95, 4.5*0.31);
\draw (0,0) -- (4.5*-0.95, 4.5*-0.31);
\draw (0,0) -- (4.5*-0.59, 4.5*-0.81);
\draw (0,0) -- (4.5*0.59, 4.5*-0.81);
\draw (0,0) -- (4.5*0.95, 4.5*-0.31);

\draw (3,0) node[above]{$\gamma_0$};
\draw (4.5*0.31, 4.5*0.95) node[above]{$\gamma_1$};
\draw (4.5*-0.81, 4.5*0.58) node[above]{$\gamma_2$};
\draw (4.5*-0.81, 4.5*-0.58) node[below]{$\gamma_3$};
\draw (4.5*0.31, 4.5*-0.95) node[below]{$\gamma_4$};

\begin{scope}[very thick, , every node/.style={sloped,allow upside down}]
\draw (0,0) -- node{\midarrow} (4.5,0); 
\draw   (4.5*0.31, 4.5*0.95) -- node {\midarrow} (0,0) ;
\draw   (4.5*-0.81, 4.5*0.58) -- node {\midarrow} (0,0) ;

\draw   (4.5*-0.81, 4.5*-0.58) -- node {\midarrow} (0,0) ;
\draw   (4.5*0.31, 4.5*-0.95) -- node {\midarrow} (0,0) ;

\end{scope}
\end{tikzpicture}
\caption{Curves of the complex plane used to build the solutions to equation (\ref{equazioneDifferenzialeGenerale}) with $n=5$.}\label{figuraCurveDimostrazioneTeorema}
	\end{minipage}\hfill
	\begin{minipage}{0.49\textwidth}
		\centering
\begin{tikzpicture}[scale = 0.77]

\draw[pattern=horizontal lines, pattern color=lightgray] (0,0) -- (4.6*0.966, 4.6*0.259) -- (4.6*0.707, 4.6*0.707) -- cycle;
\draw[pattern=horizontal lines, pattern color=lightgray] (0,0) -- (4.6*0.259, 4.6*0.966) -- (4.6*-0.259, 4.6*0.966) -- cycle;
\draw[pattern=horizontal lines, pattern color=lightgray] (0,0) -- (4.6*-0.966, 4.6*0.259) -- (4.6*-0.707, 4.6*0.707) -- cycle;
\draw[pattern=horizontal lines, pattern color=lightgray] (0,0) -- (4.6*-0.966, 4.6*-0.259) -- (4.6*-0.707, 4.6*-0.707) -- cycle;
\draw[pattern=horizontal lines, pattern color=lightgray] (0,0) -- (4.6*-0.259, 4.6*-0.966) -- (4.6*0.259, 4.6*-0.966) -- cycle;\draw[pattern=horizontal lines, pattern color=lightgray] (0,0) -- (4.6*0.966, 4.6*-0.259) -- (4.6*0.707, 4.6*-0.707) -- cycle;

		\draw[->, thick, gray] (-5,0) -- (5,0) node[below, scale = 1, black]{$\pmb{\Re}$};
		\draw[->, thick, gray] (0,-5) -- (0,5) node[left, scale = 1, black]{$\pmb{\Im}$};
%
%
%
%
%
%

\draw (4.5*0.866, 4.5*0.5) node[right]{$\gamma_0$};
\draw (4.5*0, +4.5) node[above right]{$\gamma_1$};
\draw (4.5*-0.866, 4.5*0.5) node[left]{$\gamma_2$};
\draw (4.5*-0.866, 4.5*-0.5) node[left]{$\gamma_3$};
\draw (4.5*0, -4.5) node[below right]{$\gamma_4$};
\draw   (4.5*0.866, 4.5*-0.5) node[ right]{$\gamma_5$} ;

\begin{scope}[very thick, , every node/.style={sloped,allow upside down}]
\draw (0,0) -- node{\midarrow}  (4.5*0.866, 4.5*0.5); 
\draw  (4.5*0.866, 4.5*-0.5) -- node{\midarrow} (0,0) ; 

\draw   (4.5*0, 4.5) -- node {\midarrow} (0,0) ;
\draw   (4.5*0, -4.5) -- node {\midarrow} (0,0) ;
\draw   (4.5*-0.866, 4.5*0.5) -- node {\midarrow} (0,0) ;
\draw  (4.5*-0.866, 4.5*-0.5) -- node{\midarrow} (0,0) ; 

\end{scope}
\end{tikzpicture}
\caption{Curves of the complex plane used to build the solutions to equation (\ref{equazioneDifferenzialeGenerale}) with $n=6$ and $c>0$.}\label{figuraCurvePariDimostrazioneTeorema}
\end{minipage}
\end{figure}

We note that so far we have found that on the half-lines $\gamma_\theta$, with $\theta$ in (\ref{insiemeTheta}), the expression of (\ref{limiteSulContornoSemiretta}) is $-1$ (for all $\theta\in I$), but we need to build a contour $\mathcal{C}$ where the result is $0$. Therefore, we consider the middle points of the sets composing $I$, that are $\theta_k =2k\pi/n,\ k=0,\dots,n-1$. Then, we slightly modify the notation and consider $\gamma_0$ the positive real half-line, that is parametrized as $\phi_0 (t) = t,\ t\in[0,\infty)$, and $\gamma_k$ the half-line with angle $\theta_k, k\not=0$, and entering in the origin, that is represented by $\phi_k(t) = -te^{i\theta_k},\ t\in (-\infty,0)$; see Figure \ref{figuraCurveDimostrazioneTeorema}. In light of the above considerations we readily have that
\begin{equation}\label{limiteSulContornoFinale}
e^{-cwx-(-c)^{n-1}w^n/n}\Big|_{\gamma_k \oplus \gamma_0} = (0 + 1) + (-1 + 0) = 0,
\end{equation}
and therefore we can consider the solutions, with $ k=1,\dots,n-1$,
\begin{align}
y_{k} (x)& = \int_{\gamma_k\oplus\gamma_0} e^{-cwx-(-c)^{n-1}\frac{w^n}{n}}\dif w \nonumber\\
 & = \int_0^{\infty}  e^{-ctx-(-c)^{n-1}\frac{t^n}{n}} \dif t - e^{i\frac{2k\pi}{n}}\int_0^\infty  e^{-(-c)^{n-1}\frac{t^n}{n}-ctxe^{i2k\pi/n}} \dif t,\ \ \ k=1,\dots,n-1,\label{soluzioneKesimaEquazioneDifferenzialeGenerale}
\end{align}
where the second equality comes after some calculations.

To prove that the solutions in (\ref{soluzioneKesimaEquazioneDifferenzialeGenerale}) are linearly independent, we observe that, in view of the Cauchy-Goursat theorem, $y_k$ coincides with any function $\tilde{f}_k$ of the required type and whose contour $\tilde{\mathcal{C}}$ has borders on the same section as those of the curves $\gamma_0$ and $\gamma_k$ (meaning going to infinity in the wedges of angle $\theta_0 \in\bigl(-\pi/(2n),\pi/(2n)\bigr)$ and $\theta_k \in\bigl((4k-1)\pi/(2n),(4k+1)\pi/(2n)\bigr)$ respectively). Indeed, by classical methods, it is possible to construct a series of closed path relating $\tilde{\mathcal{C}}$ and $\gamma_k \oplus\gamma_0$ showing the equality of the solutions. Therefore, if by contradiction we assume $y_h=y_k$ with $h\not=k$, then we cannot derive $n-1$ independent solutions for the equations by means of the $n$ wedges determined by the angles in the set $I$ in (\ref{insiemeTheta}) (intuitively we do not have enough "separated" wedges).

Now, from the $n-1$ independent solutions in (\ref{soluzioneKesimaEquazioneDifferenzialeGenerale}) we obtain those presented in the statement of the Theorem by calculating
\begin{equation}\label{soluzioniFinaliEquazioneDifferernzialeGenerale}
 y_k^+ (x) = \frac{y_{k}(x)+ y_{n-k}(x)}{2} \ \ \text{ and }\ \ y_k^- (x) = \frac{-y_{k}(x)+ y_{n-k}(x)}{2i}, \ \ \ \text{ natural } k\le\frac{n}{2}.
\end{equation}
Note that in the case of $n$ even, with $k = n/2$ we have $y_{n/2}^- = 0$ and $y_{n/2}^+ = y_{n/2}$ (see Example \ref{esempioSoluzionePariAsseX}).

Below we show the calculation of $y_k^-$; the case $y_k^+$follows in a similar manner. Let us consider $a_k = \cos(2k\pi/n)$ and $b_k = \sin(2k\pi/n)$,
\begin{align}
y_k^- (x) &= \frac{-y_{k}(x)+ y_{n-k}(x)}{2i} \nonumber\\
&=\frac{1}{2i} \Biggl[(a_k+ i b_k) \int^{\infty}_0 e^{-(-c)^{n-1}\frac{t^n}{n}-ctx(a_k+ib_k)}\dif t -(a_k- i b_k) \int^{\infty}_0 e^{-(-c)^{n-1}\frac{t^n}{n}-ctx(a_k-ib_k)} \dif t\Biggr]  \nonumber\\
&=\frac{1}{2i} \int^{\infty}_0 e^{-(-c)^{n-1}\frac{t^n}{n}-ctxa_k} \Bigl[ a_k\bigl(e^{-ictxb_k}-e^{ictxb_k} \bigr) + ib_k \bigl(e^{-ictxb_k}+e^{ictxb_k} \bigr) \Bigr]\dif t  \nonumber\\
& = \int^{\infty}_0 e^{-(-c)^{n-1}\frac{t^n}{n}-ctxa_k}\bigl[ -a_k \sin(ctxb_k)+ b_k \cos(ctxb_k)\bigr] \dif t\nonumber \\
& =  -\int^{\infty}_0 e^{-(-c)^{n-1}\frac{t^n}{n}-ctxa_k} \sin\Bigl(ctxb_k - \frac{2k\pi}{n}\Bigr) \dif t\label{calcoliSoluzioneEsplicitaFinaleDifferenza}
\end{align}
which coincides with (\ref{soluzioneEquazioneDifferenzialeGeneraleDifferenza}).
\\

In the case of $n$ even and $c>0$, i.e. the first case of (\ref{finePrimaCondizioneConvergenzaLimiteContornoSemiretta}), we proceed as shown above. The condition on $\theta$ is $\cos(n\theta)<0$ (it is a sufficient condition because we are not considering the possibility in $(ii)$ of (\ref{condizioneConvergenzaLimiteContornoSemiretta})), therefore we must consider
\begin{equation}\label{insiemeThetaPari}
\theta\in \bigcup_{k=0}^{n-1} \biggl(\frac{4k+1}{2n}\pi, \frac{4k+3}{2n}\pi\biggr).
\end{equation}
Here, we consider the half-lines $\gamma_0$, coming out of the origin and having angle $\theta_ 0 =\pi/n$, described by $\phi_0 (t) = te^{i\pi/n},\ t\in (0,\infty)$, and $\gamma_k$, entering in the origin and with angles $\theta_k =(2k+1)\pi/n$, described by $\phi_k (t) = -te^{i(2k+1)\pi/n},\ t\in (-\infty,0)$, for $k=1,\dots, n-1$; see Figure \ref{figuraCurvePariDimostrazioneTeorema}. Note that the angles we consider are the middle point of the intervals in (\ref{insiemeThetaPari}).
Then, we suitably adapt the solutions presented in (\ref{soluzioneKesimaEquazioneDifferenzialeGenerale}), with $k=1,\dots, n-1,$
\begin{align}
y_{k} (x)& = \int_{\gamma_k\oplus\gamma_0} e^{-cwx-(-c)^{n-1}\frac{w^n}{n}}\dif w \nonumber\\
 & = e^{i\frac{\pi}{n}}\int_0^{\infty}  e^{-ctxe^{i\pi/n}-c^{n-1}\frac{t^n}{n}} \dif t - e^{i\frac{(2k+1)\pi}{n}}\int_0^\infty  e^{-c^{n-1}\frac{t^n}{n}-ctxe^{i(2k+1)\pi/n}} \dif t.\label{soluzioneKesimaEquazioneDifferenzialeGeneralePari}
\end{align}
Note that for $k=n-1$, by dividing by $2i$ and performing some calculations similar to those in (\ref{calcoliSoluzioneEsplicitaFinaleDifferenza}), we obtain function (\ref{soluzioneEquazioneDifferenzialeGeneraleDifferenzaPari}) for $k=0$. The remaining $n-2$ linearly independent solutions given in (\ref{soluzioneEquazioneDifferenzialeGeneraleSommaPari}) and (\ref{soluzioneEquazioneDifferenzialeGeneraleDifferenzaPari}) are derived following the arguments in (\ref{soluzioniFinaliEquazioneDifferernzialeGenerale}), but setting $k=1,\dots,n/2-1$.
\end{proof}

The solutions presented in Theorem \ref{teoremaSoluzioneEquazioneDifferenzialeGenerale} can be easily expressed in terms of expected value of an exponential random variable. For instance, considering function (\ref{soluzioneEquazioneDifferenzialeGeneraleDifferenza}) and assuming that $x\in \mathbb{R}$ and $cxa_k>0$, we have
$$y_k^-(x) = - \frac{1}{cxa_k}\mathbb{E}\Biggl[ e^{-(-c)^{n-1}\frac{Y^n}{n}} \sin\Bigl( cYxb_k - \frac{2k\pi}{n}\Bigr) \Biggr], $$
where $Y\sim Exp(cxa_k)$ is an exponential random variable.

Furthermore, in order to avoid the condition depending on $x$, we can write, 
$$y_k^-(x) = - \frac{\Gamma(1/n)}{(-nc)^{1-\frac{1}{n}}}\mathbb{E}\Biggl[ e^{-c\sqrt[n]{Y}xa_k} \sin\Bigl( c\sqrt[n]{Y}xb_k - \frac{2k\pi}{n}\Bigr) \Biggr], $$
with $x\in\mathbb{C}$ and $Y\sim Gamma\Bigl(\frac{1}{n}, \frac{(-c)^{n-1}}{n}\Bigr)$ being a gamma random variable. The above results suggest that, as already proved in \cite{MO2023, ODo2012} for some particular cases, the Airy-type functions can be seen as the expectation of a damped random oscillations.

\begin{remark}[Complex constant $c$]
We point out that in the case of $c\in\mathbb{C}$, one can still use the arguments applied in the proof of Theorem \ref{teoremaSoluzioneEquazioneDifferenzialeGenerale}. In particular, we obtain a different condition on the parameter $\theta$ from the limit value (\ref{limiteSulContornoSemiretta}). Indeed, in this case we need to consider a more difficult condition. For instance we have
\begin{align*}
\exp\Bigl( -(-c)^{n-1}e^{in\theta}\Bigr) &= \exp\Bigl( -\bigl(-|c|\bigr)^{n-1}e^{i(n-1)\text{Arg}(c) + in\theta}\Bigr) \\
&= \exp\Biggl( \bigl(-|c|\bigr)^{n-1} \Bigl(\cos\bigl((n-1)\text{Arg}(c)+n\theta\bigr) + i\sin\bigl((n-1)\text{Arg}(c)+n\theta\bigr)\Bigr)\Biggr),
\end{align*}
and derive some conditions corresponding to (\ref{condizioneConvergenzaLimiteContornoSemiretta}) (note that the above exponential is not the only element that one must consider to derive the condition on $\theta$ since we did not write the term involving $e^{i\theta}$ in (\ref{limiteSulContornoSemiretta})).
\hfill$\diamond$
\end{remark}

\begin{example}\label{esempioSoluzionePariAsseX}
We show some particular cases of the solutions in Theorem \ref{teoremaSoluzioneEquazioneDifferenzialeGenerale}.
\\

$a.$ Some of the displayed functions do not involve trigonometric elements. For instance, for $n$ even and $c<0$, the solution corresponding to $k=n/2$ in (\ref{soluzioneEquazioneDifferenzialeGeneraleSomma}) reads, $x\in\mathbb{C}$,
$$y_{n/2}^+(x) =  \int_{-\infty}^\infty  e^{-(-c)^{n-1}\frac{w^n}{n}- cwx} \dif w.$$
We note that in this case the half-line denoted with $\gamma_{n/2}$ in (\ref{soluzioneKesimaEquazioneDifferenzialeGenerale}) is the negative real half-line entering in the origin.
\\

$b.$ Some cases present a very simple structure. For instance, if $n=4m$, with $m\in\mathbb{N}$, the functions (\ref{soluzioneEquazioneDifferenzialeGeneraleSomma}) and (\ref{soluzioneEquazioneDifferenzialeGeneraleDifferenza}) with $k=m$ reduces to (remember $c<0$)
\begin{align}
&y_m^+ (x) = \int_0^\infty e^{\frac{(cw)^{4m}}{4mc}} \Bigl[e^{-cwx} - \sin( cwx ) \Bigr] \dif w, \nonumber\\
&y_m^-(x) =  \int_0^\infty e^{\frac{(cw)^{4m}}{4mc}} \cos( cwx ) \ \dif w,  \label{caso4mDifferenza}
\end{align}
where we also used the fact that $\cos(x-\pi/2) = \sin(x)$. 
\\
On the other hand, if $n = 2(2m+1)$, $m\in\mathbb{N}$, the function (\ref{soluzioneEquazioneDifferenzialeGeneraleDifferenzaPari}) with $k=m$ reduces to (remember $c>0$)
\begin{equation}\label{caso4m+2Differenza}
y_m^-(x) =  \int_0^\infty e^{\frac{-(cw)^{4m+2}}{(4m+2)c}} \cos( cwx ) \ \dif w,
\end{equation}
which has the same form of the solution (\ref{caso4mDifferenza}) (where $c<0$).
\\

$c.$ For $n=5$,  the functions (\ref{soluzioneEquazioneDifferenzialeGeneraleSomma}) and (\ref{soluzioneEquazioneDifferenzialeGeneraleDifferenza}) have coefficients $a_k = \cos(2k\pi/5), \ b_k = \sin(2k\pi/5)$ for $k=1,2$, which can be expressed in terms of the golden ratio, $\varphi = (1+\sqrt{5})/2$. Indeed, the following relationships hold
\begin{align*}
\cos\frac{\pi}{5} = \frac{\varphi}{2},\ \ \ \sin\frac{\pi}{5} = \frac{\sqrt[4]{5}}{2\sqrt{\varphi}},\ \ \ \cos\frac{2\pi}{5} = \frac{1}{2\varphi},\ \ \ \sin\frac{2\pi}{5} = \frac{\sqrt[4]{5}\sqrt{\varphi}}{2}.
\end{align*}
We refer to \cite{A2017} for further details on the Airy-type equation of order $n=5$.
\end{example}

\begin{remark}\label{remarkFormaSoluzioniTipoAiry}
In the case of real $x$, the functions (\ref{soluzioneEquazioneDifferenzialeGeneraleSomma}) and (\ref{soluzioneEquazioneDifferenzialeGeneraleDifferenza}), by means of the complex integration, can be written as follows, whenever the integral converges. For $k<n/2, a^+_k = \cos\bigl(\frac{4k+1}{2n}\pi\bigr), b^+_k = \sin\bigl(\frac{4k+1}{2n}\pi\bigr)$ and $x\in\mathbb{R}$,
\begin{align}
&y_k^+ (x) = \int_0^\infty e^{-(-c)^{n-1}\frac{w^n}{n}-cwx} \dif w -  \int_0^\infty  e^{-cwxa^+_k}\cos\Bigl( cwx b^+_k +(-c)^{n-1}\frac{w^n}{n} - \frac{4k+1}{2n}\pi \Bigr) \dif w, \label{soluzioneEquazioneDifferenzialeGeneraleSommaBordoSuperioreB} \\
&y_k^-(x) = - \int_0^\infty  e^{-cwxa^+_k}\sin\Bigl( cwx b^+_k +(-c)^{n-1}\frac{w^n}{n} - \frac{4k+1}{2n}\pi\Bigr) \dif w. \label{soluzioneEquazioneDifferenzialeGeneraleDifferenzaBordoSuperioreA}
\end{align}
Alternatively, with $a^-_k = \cos\bigl(\frac{4k-1}{2n}\pi\bigr), b^-_k = \sin\bigl(\frac{4k-1}{2n}\pi\bigr)$ and $x\in\mathbb{R}$,
\begin{align}
&y_k^+ (x) = \int_0^\infty e^{-(-c)^{n-1}\frac{w^n}{n}-cwx} \dif w -  \int_0^\infty  e^{-cwxa^-_k}\cos\Bigl( cwx b^-_k -(-c)^{n-1}\frac{w^n}{n} - \frac{4k-1}{2n}\pi \Bigr) \dif w, \label{soluzioneEquazioneDifferenzialeGeneraleSommaBordoInferioreB} \\
&y_k^-(x) = - \int_0^\infty  e^{-cwxa^-_k}\sin\Bigl( cwx b^-_k -(-c)^{n-1}\frac{w^n}{n} - \frac{4k-1}{2n}\pi \Bigr) \dif w. \label{soluzioneEquazioneDifferenzialeGeneraleDifferenzaBordoInferioreA}
\end{align}
It is important to point out that the above formulas are valid if $ c a_k^{\pm}x \ge0$, meaning that the integrals converge.
We also note that (\ref{soluzioneEquazioneDifferenzialeGeneraleSommaBordoSuperioreB}) and (\ref{soluzioneEquazioneDifferenzialeGeneraleSommaBordoInferioreB}) as well as  (\ref{soluzioneEquazioneDifferenzialeGeneraleDifferenzaBordoSuperioreA}) and (\ref{soluzioneEquazioneDifferenzialeGeneraleDifferenzaBordoInferioreA}) differ only for the sign of the term $(-c)^{n-1}\frac{w^n}{n}$ in the trigonometric functions (and obviously for the coefficient $a_k^\pm$ and $b_k^\pm$).

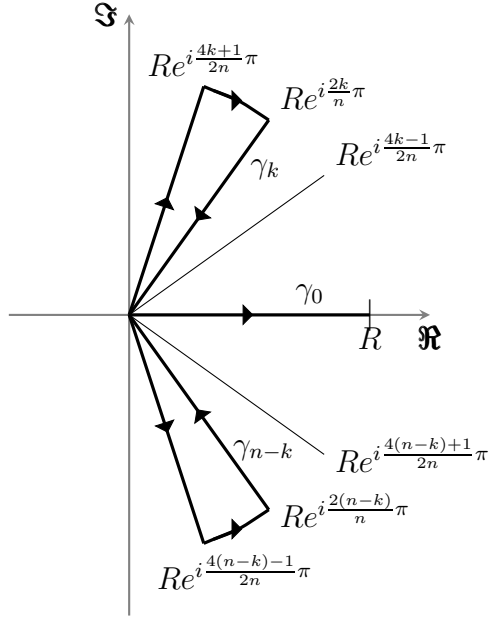
\begin{figure}
\centering
\begin{tikzpicture}[scale = 0.8]
		\draw[->, thick, gray] (-2,0) -- (5,0) node[below, scale = 1, black]{$\pmb{\Re}$};
		\draw[->, thick, gray] (0,-5) -- (0,5) node[left, scale = 1, black]{$\pmb{\Im}$};
\draw (0,0) -- (4*0.58, 4*0.81) node[above right]{$Re^{i\frac{2k}{n}\pi}$};
\draw (3.9*0.58, 3*0.81) node{$\gamma_k$};
\draw (0,0) -- (4*0.58, -4*0.81) node[ right]{$Re^{i\frac{2(n-k)}{n}\pi}$};
\draw (3.85*0.58, -3*0.75) node{$\gamma_{n-k}$};

\draw (0,0) -- (4*0.81, 4*0.58) node[above right]{$Re^{i\frac{4k-1}{2n}\pi}$};
\draw (0,0) -- (4*0.81, -4*0.58) node[ right]{$Re^{i\frac{4(n-k)+1}{2n}\pi}$};
\draw (4*0.31, 4*0.95) node[above ]{$Re^{i \frac{4k+1}{2n}\pi}$};
\draw (4*0.31, -4*0.95) node[below]{$\ \ \ \ \ \ Re^{i \frac{4(n-k)-1}{2n}\pi}$};

\draw (3,0) node[above]{$\gamma_0$};
\draw (4,0) node{$|$};
\begin{scope}[very thick, , every node/.style={sloped,allow upside down}]
\draw (0,0) -- node {\midarrow} (4,0) node[below]{$R$};
\draw  (4*0.58, 4*0.81) -- node{\midarrow} (0,0);
\draw  (0,0)  -- node {\midarrow} (4*0.31, 4*0.95);
\draw[black, very thick] plot [smooth] coordinates {(4*0.31, 4*0.95) (4*0.35, 4*0.933) (4*0.454, 4*0.891) } node{\midarrow};
\draw[black, very thick] plot [smooth] coordinates {(4*0.35, 4*0.933) (4*0.454, 4*0.891) (4*0.58, 4*0.81) };

\draw  (4*0.58, -4*0.81) -- node{\midarrow} (0,0);
\draw  (0,0)  -- node {\midarrow} (4*0.31, -4*0.95);
\draw[black, very thick] plot [smooth] coordinates {(4*0.31, -4*0.95) (4*0.35, -4*0.933) (4*0.454, -4*0.891) } node{\midarrow};
\draw[black, very thick] plot [smooth] coordinates {(4*0.35, -4*0.933) (4*0.454, -4*0.891) (4*0.58, -4*0.81) };
\end{scope}
\end{tikzpicture}
\caption{Example of the contours of integration.}\label{figuraCurvaIntegrazione}
\end{figure}


We now prove the relationship between formulas (\ref{soluzioneEquazioneDifferenzialeGeneraleDifferenzaBordoSuperioreA}) and (\ref{soluzioneEquazioneDifferenzialeGeneraleDifferenza}). First, we recall the form of the solution $y_k^-$ in the proof of Theorem \ref{teoremaSoluzioneEquazioneDifferenzialeGenerale}, that is, with natural $k \le n/2$,
\begin{align}
y_k^- (x)& = \frac{-y_{k}(x)+ y_{n-k}(x)}{2i} \nonumber\\
&= \frac{1}{2i} \Biggl[ - \int_{\gamma_k}  e^{-cwx-(-c)^{n-1}\frac{w^n}{n}}\dif w +\int_{\gamma_{n-k}}  e^{-cwx-(-c)^{n-1}\frac{w^n}{n}}\dif w \Biggr]\label{formulazioneSoluzionePerFormaAiry}
\end{align}
where we used the first equality in formula (\ref{soluzioneKesimaEquazioneDifferenzialeGenerale}) for $y_k$ and $y_{n-k}$ and $\gamma_k, \gamma_{n-k}$ defined before (\ref{limiteSulContornoFinale}).
Now, we observe that, if $c\cos\bigl(\frac{4k+1}{2n}\pi\bigr)x >0$, then, according to condition $(ii)$ of (\ref{condizioneConvergenzaLimiteContornoSemiretta}), we can include the angle $\theta = \frac{4k+1}{2n}\pi$ in the set $I$, given in (\ref{insiemeTheta}). This means that the function in the integrals of (\ref{formulazioneSoluzionePerFormaAiry}) does not have any pole in the complex region 
$$\Bigg\{z\in\mathbb{C} \,:\, Arg(z)\in \biggl(\frac{4k-1}{2n}\pi, \frac{4k+1}{2n}\pi\biggr]\Bigg\}$$
and therefore, we can use the Cauchy-Goursat theorem to rewrite the integral on $\gamma_k$. In detail, we integrate on the contour given in Fig. \ref{figuraCurvaIntegrazione} (consider the positive imaginary side of the complex plane), deriving
\begin{align}
-\int_{\gamma_k}&  e^{-cwx-(-c)^{n-1}\frac{w^n}{n}}\dif w\nonumber \\
&= \lim_{R\rightarrow\infty}  \int_{\frac{2k\pi}{n}}^{\frac{4k+1}{2n}\pi}  e^{-cx(Re^{i\theta})-(-c)^{n-1}\frac{\bigl(Re^{i\theta}\bigr)^n}{n}} iRe^{i\theta}\dif \theta \nonumber \\
&\ \ \ + \lim_{R\rightarrow\infty}  \int_0^R  \exp\Biggl(-cx(te^{i\frac{4k+1}{2n}\pi})-(-c)^{n-1}\frac{\bigl(te^{i\frac{4k+1}{2n}\pi}\bigr)^n}{n}\Biggr)\, e^{i\frac{4k+1}{2n}\pi}\dif t\nonumber\\
& = 0+ e^{i\frac{4k+1}{2n}\pi} \int_0^\infty  \exp\Bigl(-cxte^{i\frac{4k+1}{2n}\pi}-i(-c)^{n-1}\frac{t^n}{n}\Bigr) \dif t \nonumber\\
& = (a_k^+ + i b_k^+) \int_0^\infty  \exp\Biggl(-cxta_k^+-i\Bigl[ctxb_k^+ +(-c)^{n-1}\frac{t^n}{n}\Bigr]\Biggr) \dif t.\label{integraleDopoTrasformazioneCurva}
\end{align}
We point out that the first limit converges to $0$ since for all $R>0$,
\begin{align*} 
\int_{\frac{2k\pi}{n}}^{\frac{4k+1}{2n}\pi} \Big|e^{-cx(Re^{i\theta})-(-c)^{n-1}\frac{\bigl(Re^{i\theta}\bigr)^n}{n}} iRe^{i\theta}\Big| \dif\theta \le \int_{\frac{2k\pi}{n}}^{\frac{4k+1}{2n}\pi} \Big|e^{-cxR\cos\theta-(-c)^{n-1}\frac{R^n\cos(n\theta)}{n}} R\Big | \dif \theta <\infty,
\end{align*}
where the convergence is assured thanks to the conditions on $\theta$ in (\ref{condizioneConvergenzaLimiteContornoSemiretta}).

Then, we use the above arguments to rewrite the integral on $\gamma_{n-k}$. It is important to note that (remember $x$ real) $cx\cos\bigl(\frac{4k+1}{2n}\pi\bigr) \ge0 \iff cx\cos\bigl(\frac{4(n-k)-1}{2n}\pi\bigr) \ge0$ and this is simply verified because 
$$\cos\biggl(\frac{4(n-k)-1}{2n}\pi\biggr) = \cos\biggl(2n\pi-\frac{4k+1}{2n}\pi\biggr) = \cos\biggl(\frac{4k+1}{2n}\pi\biggr).$$
Furthermore, we also observe that $e^{i\frac{4(n-k)+1}{2n}\pi} = \cos\bigl(\frac{4k+1}{2n}\pi\bigr) + i \sin\bigl(-\frac{4k+1}{2n}\pi\bigr) =a_k^+ - ib_k^+$ which is the conjugate of the complex number appearing in (\ref{integraleDopoTrasformazioneCurva}). Thus, we obtain
\begin{equation}\label{secondoIntegraleDopoTrasformazioneCurva}
\int_{\gamma_{n-k}}  e^{-cwx-(-c)^{n-1}\frac{w^n}{n}}\dif w = -(a_k^+ - i b_k^+) \int_0^\infty  \exp\Biggl(-cxta_k^+ +i\biggl[ctxb_k^+ +(-c)^{n-1}\frac{t^n}{n}\biggr]\Biggr) \dif t.
\end{equation}
Now, by inserting (\ref{integraleDopoTrasformazioneCurva}) and (\ref{secondoIntegraleDopoTrasformazioneCurva}) in (\ref{formulazioneSoluzionePerFormaAiry}) we readily arrive at result (\ref{soluzioneEquazioneDifferenzialeGeneraleDifferenzaBordoSuperioreA})
\begin{align*}
y_k^- (x)& = \frac{1}{2i} \int_0^\infty e^{-ctxa_k^+} \Biggl[ a_k^+ \biggl( e^{-i\bigl[ctxb_k^+ +(-c)^{n-1}\frac{t^n}{n}\bigr]} - e^{i\bigl[ctxb_k^+ +(-c)^{n-1}\frac{t^n}{n}\bigr]} \biggr) \\
& \ \ \ + ib_k^+ \biggl( e^{-i\bigl[ctxb_k^+ +(-c)^{n-1}\frac{t^n}{n}\bigr]} + e^{i\bigl[ctxb_k^+ +(-c)^{n-1}\frac{t^n}{n}\bigr]} \biggr)\Biggr] \dif t \\
& = \int_0^\infty e^{-ctxa_k^+}  \Biggl[ - a_k^+\sin\biggl( ctxb_k^+ +(-c)^{n-1}\frac{t^n}{n}\biggr) + b_k^+ \cos\biggl( ctxb_k^+ +(-c)^{n-1}\frac{t^n}{n}\biggr) \Biggr] \dif t \\
& =  - \int_0^\infty  e^{-ctxa^+_k}\sin\Bigl( ctx b^+_k +(-c)^{n-1}\frac{t^n}{n} - \frac{4k+1}{2n}\pi\Bigr) \dif t.
\end{align*}

The proof of the formulas (\ref{soluzioneEquazioneDifferenzialeGeneraleSommaBordoSuperioreB}), (\ref{soluzioneEquazioneDifferenzialeGeneraleSommaBordoInferioreB}) and (\ref{soluzioneEquazioneDifferenzialeGeneraleDifferenzaBordoInferioreA}) are derived in a similar way.

We point out that these are not the only alternative integral forms of these class of functions, indeed one could suitably modify the contour of integration. 
Furthermore, similar computation can be equivalently performed for solutions of the differential equation (\ref{equazioneDifferenzialeGenerale}) with $n$ even and $c>0$, given in (\ref{soluzioneEquazioneDifferenzialeGeneraleSommaPari}) and (\ref{soluzioneEquazioneDifferenzialeGeneraleDifferenzaPari}).
\hfill$\diamond$
\end{remark}

\begin{example}
It is interesting to observe that the above representations are closer to the classical form of the hyper-Airy functions. It is easy to verify that for $n = 2m+1,\ m\in\mathbb{N}$, if $m$ is even then (\ref{soluzioneEquazioneDifferenzialeGeneraleDifferenzaBordoSuperioreA}), by setting $k = m/2$ and  using that $\sin(x + \pi/2) = \cos(x)$, reduces a hyper-Airy type function (\ref{funzioneAiryReale}) with $\alpha = 2m+1$. Indeed, $a_{m/2}^+ =\cos\bigl(\frac{4m/2+1}{2n}\pi\bigr) = \cos(\pi/2) = 0, b^+_k  =\sin(\pi/2) = 1$,
\begin{align}
y_{m/2}^-(x) &= - \int_0^\infty  \sin\Bigl( cwx +(-c)^{2m}\frac{w^{2m+1}}{2m+1} - \frac{\pi}{2}\Bigr) \dif w \nonumber\\
& =  \int_0^\infty  \cos\Bigl( cwx +c^{2m}\frac{w^{2m+1}}{2m+1} \Bigr) \dif w,\ \ \ m \text{ even, } \label{iperAirySoluzioneEquazioneGenerale}
\end{align}
which resembles (\ref{funzioneAiryReale}). Finally, by considering $c = (-1)^m = 1$ (since $m$ is even) we obtain (\ref{funzioneAiryReale}) with $\alpha = 2m+1$ and this solves the differential equation (\ref{equazioneAiryDispari}). We point out that in the case of $m$ odd one must use formula (\ref{soluzioneEquazioneDifferenzialeGeneraleDifferenzaBordoInferioreA}) for the sake of convergence and set $k = (m+1)/2$, obtaining 
\begin{equation}\label{iperAirySoluzioneEquazioneGeneraleCasoDispari}
y_{\frac{m+1}{2}}^-(x) =  \int_0^\infty \cos\Bigl( cwx -c^{2m}\frac{w^{2m+1}}{2m+1} \Bigr) \dif w,\ \ \ m \text{ odd. }
\end{equation}
\hfill$\diamond$
\end{example}

\begin{remark}
We point out that the functions in Theorem \ref{teoremaSoluzioneEquazioneDifferenzialeGenerale} are not all integrable on both half-lines. It is useful to remind that, for instance, also $\int_0^\infty Bi(x) \dif x =\infty$, with $Bi$ the Airy function defined in (\ref{funzioniAiryClassiche}). Now, we show how to perform the integral of functions of the type of (\ref{soluzioneEquazioneDifferenzialeGeneraleDifferenza}) on the positive half-line, but the technique can be used also for the other possible functions (whose integral converges).
\\Assume $\theta\in[0,\pi]$ and $a_\theta = \cos\theta, b_\theta=\sin\theta $ such that $ca_\theta \ge0$,
\begin{align}
 - \int_0^\infty& \dif x \int_0^\infty e^{-(-c)^{n-1}\frac{t^n}{n}-ctxa_\theta}\sin(ctxb_\theta - \theta)\dif t \nonumber\\
&= \int_0^\infty \lim_{\mu \downarrow 0} e^{-\mu x} \dif x \int_0^\infty  e^{-(-c)^{n-1}\frac{t^n}{n}-ctxa_\theta} \Im\Bigl(e^{-i(ctxb_\theta-\theta)}\Bigr) \dif t \nonumber\\
& = \int_0^\infty e^{-(-c)^{n-1}\frac{t^n}{n}}\dif t \int_0^\infty  \lim_{\mu \downarrow 0} e^{-\mu x} \Im\Bigl(e^{-ctxe^{i\theta}+i\theta}\Bigr) \dif x \nonumber\\
& = \int_0^\infty e^{-(-c)^{n-1}\frac{t^n}{n}}\dif t \lim_{\mu \downarrow 0}  \Im\Biggl( \int_0^\infty  e^{-\mu x-ctxe^{i\theta}+i\theta} \dif x \Biggr) \label{primoPassaggioTCDCalcoloIntegrale}\\
& = \int_0^\infty e^{-(-c)^{n-1}\frac{t^n}{n}}  \lim_{\mu \downarrow 0}  \Im\Biggl( \frac{e^{i\theta}}{\mu+cte^{i\theta}}\, \frac{\mu + cte^{-i\theta}}{\mu + cte^{-i\theta}}\Biggr)   \dif t \nonumber\\
& = \int_0^\infty e^{-(-c)^{n-1}\frac{t^n}{n}}  \lim_{\mu \downarrow 0}  \Im\Biggl( \frac{\mu e^{i\theta} + ct}{\mu^2+c^2t^2+2\mu ct\cos\theta}\Biggr) \dif t \nonumber\\
& = \int_0^\infty e^{-(-c)^{n-1}\frac{t^n}{n}}  \lim_{\mu \downarrow 0}  \frac{\mu \sin\theta}{\mu^2+c^2t^2+2\mu ct\cos\theta} \dif t \nonumber\\
& =  \lim_{\mu \downarrow 0} \int_0^\infty e^{-(-c)^{n-1}\frac{(\mu s)^n}{n}}  \frac{\mu \sin\theta}{\mu^2+c^2\mu^2s^2+2\mu^2 cs\cos\theta} \mu\dif s \label{secondoPassaggioTCDCalcoloIntegrale}\\
& = \int_0^\infty \lim_{\mu \downarrow 0} e^{-(-c)^{n-1}\frac{(\mu s)^n}{n}}  \frac{\sin\theta}{1+c^2s^2+ cs\cos\theta} \dif s \label{terzoPassaggioTCDCalcoloIntegrale} \\
& = \int_0^\infty \frac{\sin\theta}{(cs+ \cos\theta)^2 + \sin(\theta)^2} \dif s \nonumber \\
& = \int_{\text{cotg}\theta}^{\text{sgn(c)}\infty} \frac{1}{c} \frac{\dif w}{w^2+1} \label{quartoPassaggioTCDCalcoloIntegrale} \\
& =  -  \frac{1}{c} \text{arccotg}(w)\Big|_{\text{cotg}\theta}^{\text{sgn}(c)\infty} \nonumber\\
& = \begin{cases}\begin{array}{l l}
\frac{\pi-\theta}{-c}, &c<0,\\
\frac{\theta}{c}, & c>0.
\end{array}\end{cases} \label{integraleGeneraleSemiAssePositivoSoluzioneTipoA}
\end{align}
We used the dominated convergence theorem in steps (\ref{primoPassaggioTCDCalcoloIntegrale}), (\ref{secondoPassaggioTCDCalcoloIntegrale}) and (\ref{terzoPassaggioTCDCalcoloIntegrale}). In (\ref{secondoPassaggioTCDCalcoloIntegrale}) we also used the change of variable $\mu s = t$, while in step (\ref{quartoPassaggioTCDCalcoloIntegrale}) we used the change of variable $cs + cos\theta = w\sin\theta$. We also point out that we have never exchange the limit and the imaginary part, in fact one must be very careful since the real part of the complex number may diverge.
\\
Note that the result (\ref{integraleGeneraleSemiAssePositivoSoluzioneTipoA}) does not depend on the parameter $n$, but this is required to be greater than $1$ for the need of convergence.

If $n$ odd we can use result (\ref{integraleGeneraleSemiAssePositivoSoluzioneTipoA}) to obtain the integral on the negative half-line by performing the change of variable $y = -x$. Indeed, we have, with $ca_\theta\le0$,
\begin{align}
- \int_{-\infty}^0 & \dif x \int_0^\infty e^{-(-c)^{n-1}\frac{t^n}{n}-ctxa_\theta}\sin(ctxb_\theta - \theta)\dif t \nonumber\\
& = - \int_0^\infty  \dif y \int_0^\infty e^{-(-c)^{n-1}\frac{t^n}{n}-(-c)tya_\theta}\sin\bigl((-c)tyb_\theta - \theta\bigr)\dif t \nonumber\\
& = \begin{cases}\begin{array}{l l}
\frac{\pi-\theta}{c}, &c>0,\\
\frac{\theta}{-c}, & c<0,
\end{array}\end{cases} \label{integraleGeneraleSemiAsseNegativoSoluzioneTipoA}
\end{align}
where in the last step we used (\ref{integraleGeneraleSemiAssePositivoSoluzioneTipoA}) with $-c$.

If $a_\theta = 0$ (and therefore $\theta = \pm\pi/2$ and $b_\theta = \pm1$), then in light of (\ref{integraleGeneraleSemiAssePositivoSoluzioneTipoA}) and (\ref{integraleGeneraleSemiAsseNegativoSoluzioneTipoA}) we have the following integral on the real line
$$-\int_{-\infty}^\infty \dif x \int_0^\infty e^{-(-c)^{n-1}\frac{t^n}{n}}\sin\biggl(\pm<t< ctx \mp \frac{\pi}{2}\biggr)\dif t = \frac{\pi}{|c|},\ \ \ n \text{ odd},\ c\in\mathbb{R}\setminus\{0\}, \ \theta \in[0,\pi],$$
which does not depend on $n$ and $\theta$.

By setting $\theta = 2k\pi/n$ in (\ref{integraleGeneraleSemiAssePositivoSoluzioneTipoA}) we readily obtain the corresponding integral of (\ref{soluzioneEquazioneDifferenzialeGeneraleDifferenza}),
\begin{equation*}
\int_0^\infty y_k^-(x)\dif x = 
\begin{cases}\begin{array}{l l}
\frac{n-2k}{-nc}\pi, &c<0,\\
\frac{2k}{nc}\pi, & c>0.
\end{array}\end{cases}
\end{equation*}
\hfill$\diamond$
\end{remark}

\begin{corollary}[Power series of the solutions]\label{corollarioSeriePotenzeSoluzioneEquazioneDifferenzialeGenerale}
The functions (\ref{soluzioneEquazioneDifferenzialeGeneraleSomma}) and (\ref{soluzioneEquazioneDifferenzialeGeneraleDifferenza}) ($n$ odd or $n$ even and $c<0$) admit the following power series representation, with $x\in\mathbb{C}$ and natural $k \le n/2$,
\begin{align}
&y_k^+ (x) = \sum_{j=0}^\infty \frac{(-\text{sgn}(c)x)^j}{j!} \Biggl[1-\cos\Bigl((j+1)\frac{2k\pi}{n} \Bigr) \Biggr]\Gamma\Bigl(\frac{j+1}{n}\Bigr)(n|c|)^{\frac{j+1}{n}-1} , \label{serieSoluzioneEquazioneDifferenzialeGeneraleSomma} \\
&y_k^-(x) =  \sum_{j=0}^\infty \frac{(-\text{sgn}(c)x)^j}{j!} \sin\Bigl((j+1)\frac{2k\pi}{n} \Bigr) \Gamma\Bigl(\frac{j+1}{n}\Bigr)(n|c|)^{\frac{j+1}{n}-1}. \label{serieSoluzioneEquazioneDifferenzialeGeneraleDifferenza}
\end{align}
The functions (\ref{soluzioneEquazioneDifferenzialeGeneraleSommaPari}) and (\ref{soluzioneEquazioneDifferenzialeGeneraleDifferenzaPari}) ($n$ even and $c>0$) admit the following power series representation, with $x\in\mathbb{C}$,
\begin{align}
&y_k^+ (x) = \sum_{j=0}^\infty  \frac{(-x)^j}{j!} \Biggl[e^{i\frac{\pi}{n}(j+1)}-\cos\Bigl((j+1)\frac{(2k+1)\pi}{n} \Bigr) \Biggr]\Gamma\Bigl(\frac{j+1}{n}\Bigr)(nc)^{\frac{j+1}{n}-1}, \ \, k=1,\dots,\frac{n}{2}-1, \label{serieSoluzioneEquazioneDifferenzialeGeneraleSommaPari} \\
&y_k^-(x) = \sum_{j=0}^\infty \frac{(-x)^j}{j!} \sin\Bigl((j+1)\frac{(2k+1)\pi}{n} \Bigr) \Gamma\Bigl(\frac{j+1}{n}\Bigr)(nc)^{\frac{j+1}{n}-1}, \ \ \ k=0,\dots,\frac{n}{2}-1. \label{serieSoluzioneEquazioneDifferenzialeGeneraleDifferenzaPari}
\end{align}
\end{corollary}

Note that the coefficients of the power series in Corollary \ref{corollarioSeriePotenzeSoluzioneEquazioneDifferenzialeGenerale} satisfy a "cyclic" relationship which can be easily obtained from the differential equation (\ref{equazioneDifferenzialeGeneraleDerivate}) concerning the derivative functions. In fact, by setting $x = 0$ we have $f_n^{(n+m-1)}(0) = -c m f_n^{(m-1)} (0)$.
\\

We observe that the power series (\ref{serieSoluzioneEquazioneDifferenzialeGeneraleDifferenza}), coincides with the result (15) in \cite{MO2023} by setting $n\leftarrow 2n+1$, $k=n/2$ and $c = (-1)^n$. In particular, the case for $n$ odd is easy to show, while in the case of $n$ even one must perform some quite cumbersome calculations (we suggest to suitably modify the integral computation (18) and at the end of the proof of \cite{MO2023}).

\begin{proof}
We prove formula (\ref{serieSoluzioneEquazioneDifferenzialeGeneraleDifferenza}) concerning the case of $n$ odd or $n$ even and $c<0$. 

First of all we observe that the expression (\ref{soluzioneEquazioneDifferenzialeGeneraleDifferenza}) easily permits the interchange between the derivative and the integral, therefore we can write down the power series expansion centered in $x=0$. Therefore, we obtain the derivative of any order by means of some calculations, for $j\in\mathbb{N}$ and $a_k = \cos(2k\pi/n), b_ = \sin(2k\pi/n),$
\begin{equation}
\frac{\dif^j}{\dif x^j} y_k^-(x) = - (-c)^j \int_0^\infty t^j
 e^{-(-c)^{n-1}\frac{t^n}{n}-ctxa_k} \sin\Bigl(b_kctx - \frac{2k\pi}{n}(j+1)\Bigr)\dif t.
\end{equation}
We note that $(-c)^{n-1} = |c|^{n-1}\ge0$ because of the conditions on $n$ and $c$. Now, we determine the value of the derivative in $x=0$ by suitably applying the Euler's Gamma function.
\begin{align}\label{valoreDerivataSoluzioneInZero}
\frac{\dif^j}{\dif x^j} y_k^-(0) &= - (-c)^j \int_0^\infty t^j
 e^{-|c|^{n-1}\frac{t^n}{n}} \sin\Bigl(- \frac{2k\pi}{n}(j+1)\Bigr)\dif t \\
& = (-c)^j \sin\Bigl(- \frac{2k\pi}{n}(j+1)\Bigr)\int_0^\infty  \frac{\dif s}{n} \,s^{\frac{j}{n}-1}
 e^{-|c|^{n-1}\frac{s}{n}} \nonumber\\
&= \frac{(-c)^j}{n} \sin\Bigl(- \frac{2k\pi}{n}(j+1)\Bigr) \Gamma\Bigl(\frac{j+1}{n}\Bigr) \Bigl(\frac{|c|^{n-1}}{n}\Bigr)^{-\frac{j+1}{n}}\nonumber\\
&=\bigl(-\text{sgn}(c)\bigr)^j  (n|c|)^{\frac{j+1}{n}-1}\sin\Bigl(- \frac{2k\pi}{n}(j+1)\Bigr) \Gamma\Bigl(\frac{j+1}{n}\Bigr). \label{valoreDerivataSoluzioneInZero}
\end{align}
Thanks to the values in (\ref{valoreDerivataSoluzioneInZero}) we readily arrive at the power series representation (\ref{serieSoluzioneEquazioneDifferenzialeGeneraleDifferenza}).

We point out that the proof of the other power series in the statement follow the same arguments.
\end{proof}

\begin{remark}[Differential equation for the derivatives]\label{remarkEquazioneDerivateSoluzioni}
Equation (\ref{equazioneDifferenzialeGenerale}) has a very convenient form to study the differential equations of the derivatives of its solutions. Indeed, by deriving $m\in\mathbb{N}$ times, we obtain
\begin{align*}
0 =y^{(n+m-1)} + c\frac{\dif^m }{\dif x^m}(xy) = \Bigl(y^{(m)}\Bigr)^{(n-1)} + cx y^{(m)} + m c y^{(m-1)},
\end{align*}
and we readily arrive at the non-homogeneous differential equation
\begin{equation}\label{equazioneDifferenzialeGeneraleDerivate}
y^{(n-1)} + cxy = -m c f^{(m-1)}, \ \ \ x\in \mathbb{C}, \ \text{integers } m\ge0,n\ge 2,\ c\in \mathbb{R}\setminus{\{0\}},
\end{equation}
satisfied by the function $f^{(m)}$, with $f$ satisfying (\ref{equazioneDifferenzialeGenerale}). 
\hfill$\diamond$
\end{remark}

\begin{remark}[Scorer-type differential equations]
The results on the Airy-type functions also permit us to present solutions to the higher order Scorer-type differential equations \cite{S1950}. In detail, we prove that both terms in the right-hand side of (\ref{soluzioneEquazioneDifferenzialeGeneraleSomma}) and of (\ref{soluzioneEquazioneDifferenzialeGeneraleSommaPari}) satisfy the equation
\begin{equation}\label{equazioneIperScorer}
y^{(n-1)} + cxy = 1, \ \ \ x\in\mathbb{C},
\end{equation}
which in the case of $n = 3$ and $c = -1$ reduces to the classical Scorer equation.

The previous result, in the case of function (\ref{soluzioneEquazioneDifferenzialeGeneraleSomma}), follows by proving that the first term of the right-hand side of (\ref{soluzioneEquazioneDifferenzialeGeneraleSomma}) satisfies the non-homogeneous differential equation (\ref{equazioneIperScorer}). Indeed, for this term we can exchange the derivative and the integral easily showing that 
\begin{align*}
\Bigl( \frac{\dif^{n-1}}{\dif x^{n-1}} +cx \Bigr)&\int_0^\infty e^{-(-c)^{n-1}\frac{w^n}{n}-cwx} \dif w \nonumber\\
& = \int_0^\infty \bigl((-cw)^{n-1} + cx\bigr) e^{-(-c)^{n-1}\frac{w^n}{n}-cwx} \dif w \nonumber\\
& = -e^{-(-c)^{n-1}\frac{w^n}{n}-cwx} \Big|_0^\infty  = 1,
\end{align*}
where we need to keep in mind that $n$ is odd or $n$ is even and $c<0$. A similar argument can be used also for the case of the elements in the right-hand side of function (\ref{soluzioneEquazioneDifferenzialeGeneraleSommaPari}) regarding the case $n$ even and $c>0$.

Hence, the following functions satisfy the higher-order Scorer-type differential equation (\ref{equazioneIperScorer}) 
\begin{itemize}
\item if $n$ odd or $n$ even and $c<0$,
\begin{align}
&f(x) = \int_0^\infty e^{-(-c)^{n-1}\frac{w^n}{n}-cwx} \dif w,\\
& g_k(x) = \int_0^\infty e^{-(-c)^{n-1}\frac{w^n}{n}-cwxa_k}  \cos\Bigl( cwx b_k - \frac{2k\pi}{n} \Bigr)\dif w,\label{casoAiryScorer}
\end{align}
with natural $k \le n/2$ and $a_k = \cos(2k\pi/n)$ and $b_k = \sin(2k\pi/n)$;
\item if $n$ even and $c>0$,
\begin{align}
&f(x) = e^{i\frac{\pi}{n}}\int_0^\infty e^{(-c)^{n-1}\frac{w^n}{n}-cwxe^{i\pi/n}} \dif w, \ \ \ k=1,\dots,\frac{n}{2}-1\\
&g_k(x) = \int_0^\infty e^{(-c)^{n-1}\frac{w^n}{n}-cwxa_k}  \cos\Bigl( cwx b_k - \frac{2k+1}{n}\pi \Bigr)\dif w,k=0,\dots,\frac{n}{2}-1,
\end{align}
with $a_k = \cos\bigl((2k+1)\pi/n\bigr)$ and $b_k = \sin\bigl((2k+1)\pi/n\bigr)$.
\end{itemize}

We point out that in the case of (\ref{equazioneIperScorer}) with $n=3$ and $c=-1$ we obtain the original Scorer's equation (disregarding the factor $1/\pi$) and its corresponding solution (for $k=1$) which can be written as
\begin{align*}
g_1(x) &=  \int_0^\infty e^{-\frac{w^3}{3}+\frac{wx}{2}} \cos\Bigl(\frac{\sqrt{3}}{2}wx + \frac{2\pi}{3}\Bigr) \dif w\\
& = \int_0^\infty \sin\Bigl(wx + \frac{w^3}{3}\Bigr)\dif w, 
\end{align*}
where the second equality follows by considering that, in light of Remark \ref{remarkFormaSoluzioniTipoAiry}, function (\ref{casoAiryScorer}) coincides with the second term of function (\ref{soluzioneEquazioneDifferenzialeGeneraleSommaBordoInferioreB}) with $n=3,k=1, c=-1$ (and also (\ref{soluzioneEquazioneDifferenzialeGeneraleSommaBordoSuperioreB}),) when the integrals converge.
\hfill$\diamond$

We point out that, in light of Remark \ref{remarkFormaSoluzioniTipoAiry}, function (\ref{casoAiryScorer}) coincides with the second term of functions (\ref{soluzioneEquazioneDifferenzialeGeneraleSommaBordoSuperioreB}) and (\ref{soluzioneEquazioneDifferenzialeGeneraleSommaBordoInferioreB}) (when the integrals converge). From the latter, if 
$$ g_1(x) =  -  \int_0^\infty  e^{-cwxa^-_k}\cos\Bigl( cwx b^-_k -(-c)^{n-1}\frac{w^n}{n} - \frac{4k-1}{2n}\pi \Bigr) \dif w = \int_0^\infty \sin\Bigl(wx + \frac{w^3}{3}\Bigr)\dif w, $$
where in the last equality we performed the required substitutions and some calculation.
\hfill$\diamond$
\end{remark}

\begin{remark}
We note that for the classical Airy second-order equation (i.e. (\ref{equazioneDifferenzialeGenerale}) with $n=3$ and $c = -1$), the three possible products of $\text{Ai}$ and $\text{Bi}$ (meaning $\text{Ai}^2, \text{Bi}^2$ and $\text{Ai}\text{Bi}$, see (\ref{funzioniAiryClassiche})) are the linearly independent solutions of the third-order differential equation
$$ y^{(3)} - 4xy' -2y = 0.$$

Furthermore, in the case of the product of three and four functions (denoted by $f_j$) solving the classical Airy functions we have the following differential equations,
\begin{align*}
&y = f_1f_2f_3 \ \ \text{ satisfies }\ \  y^{(4)} = - 9x^2y +10 y' + 10x y'',\\
&y = f_1f_2f_3f_4\ \  \text{ satisfies }\ \  y^{(5)} = -64xy -64x^2 y' + 30x y'' + 20xy'''.
\end{align*}
With the above equations at hand, one can derive the other cases. However, we have not been able to obtain a nice general differential equation for the product of $n$ solutions to the classical Airy functions, but it is clear that it must be of order $n+1$, which is the number of possible combinations with repetitions of $2$ elements (the number of linearly independent solutions to the Airy function) in $n$ places.

We refer to the series of papers of Reid \cite{R1995, R1997, R1997b} for some explicit form of the products involving the classical Airy functions and \cite{V2010} for further results.
\\

We make similar considerations for further cases as well. For instance, in the case of equation (\ref{equazioneDifferenzialeGenerale}) with $n=4$ and $c = -1$, the products of two possible solutions satisfies the following sixth-order equation
$$xy^{(6)} -y^{(5)} -7x^2y^{(3)} -7xy^{(2)}+7y'-8x^3y = 0$$
and these products are its linear independent solutions.

In general, we conjecture that the products of the solutions of the equation (\ref{equazioneDifferenzialeGenerale}) are the independent solutions to a differential equation of order $\binom{n}{2} = n(n-1)/2$, which is equal to the number of the possible products of the $n-1$ independent solutions.
\hfill$\diamond$
\end{remark}

We conclude this section by studying the differential equation governing the convolution of the solutions to (\ref{equazioneDifferenzialeGenerale}).
\begin{proposition}
Let $f_{n}$ be a solution to equation (\ref{equazioneDifferenzialeGeneraleIntro}) with $c = c_n$. Then, the (real) convolution $g = \Conv_{j=1}^N f_{n_j} : \mathbb{R}\longrightarrow\mathbb{R}$ satisfies
\begin{equation}\label{equazioneConvoluzioneGenerale}
\sum_{j=1}^N y^{(n_j - 1)} \prod_{h=1 h\not=j}^N c_{n_h} + x y \prod_{j=1}^N c_{n_j} = 0.
\end{equation}
\end{proposition}

\begin{proof}
We prove the statement for $N=2$. Let us use the notation $n=n_1, m=n_2$ and $g = f_n\ast f_m$, then
\begin{align}
\frac{\dif^{n-1}}{\dif x^{n-1}}  \int_{-\infty}^\infty f_{n}(x-w) f_m(w)\dif w &= -c_n \int_{-\infty}^\infty (x-w)f_n(x-w)f_m(w)\dif w  \nonumber\\
& = -c_n x f_n\ast f_m(x)+ c_n \int_{-\infty}^\infty w f_n(x-w)f_m(w) \dif w\nonumber\\
& = -c_n x g+ c_n \int_{-\infty}^\infty (x-z) f_n(z) f_m(x-z) \dif z  \label{cambioVariabileConvoluzione} \\
& = -c_n x g-\frac{c_n}{c_m} \frac{\dif^{m-1}}{\dif x^{m-1}} g, \nonumber
\end{align}
where in the first step, after exchanging derivative and integral, we used that $f_n$ satisfies (\ref{equazioneDifferenzialeGeneraleIntro}); in step (\ref{cambioVariabileConvoluzione}) we used the change of variable $z = x-w$ and in the last step we considered that $f_m$ satisfies (\ref{equazioneDifferenzialeGeneraleIntro}) and exchanged again derivative and integral.
By rearranging the above relationship we obtain the corresponding case of (\ref{equazioneConvoluzioneGenerale}) for $N=2$,
$$ c_m y^{(n-1)} + c_n y^{(m-1)} +c_n c_m y = 0.$$

The statement can be proved for a general $N$ by induction. The steps coincide with those above so they are omitted.
\end{proof}

Similarly to Remark \ref{remarkEquazioneDerivateSoluzioni}, the $m$-th derivative, $g^{(m)}$, of the convolution satisfy the non-homogeneous differential equation
\begin{equation}\label{equazioneConvoluzioneGenerale}
\sum_{j=1}^N y^{(n_j - 1)} \prod_{\substack{h=1 \\h\not=j}}^N c_{n_h} + x y \prod_{j=1}^N c_{n_j} = - mg^{(m-1)}\prod_{j=1}^N c_{n_j}.
\end{equation}

In the next section we show that the following generalization of the hyper-Airy function is solution to equation (\ref{equazioneConvoluzioneGenerale}), $N\in \mathbb{N}$ and $\alpha_1, \dots,\alpha_N>1$,
\begin{equation}\label{generalizzazioneFunzioneAiry}
\mathcal{A}(x) = \int_0^\infty \cos\Biggl(xw + \sum_{k=1}^N \frac{a_k w^{\alpha_k}}{\alpha_k} \Biggr)\dif w, \ \ \ x\in\mathbb{C},
\end{equation}
with $a_1,\dots,a_N\not=0$.

\section{Heat-type equations and pseudo-processes}

By following the line of previous papers, see for instance \cite{CO2024b, MO2023, OT2014}, we consider the family of fractional-order pseudo-processes related to the generalization of higher-order heat-type equations. We refer also to \cite{BKOS2007} and references therein for further details on higher-order heat-equations.

First of all, we recall that, as highlighted by Askari and Ansari \cite{AA2020}, the hyper-Airy functions (\ref{funzioneAiryReale}) (with $\alpha = 2m+1,\ n\in\mathbb{N}$) are related to the following heat-type differential problem.
\begin{equation}\label{problemaDifferenzialeIperAiryIntere}
\begin{cases}
\frac{\partial u}{\partial t}(t,x) = (-1)^n \frac{\partial^{2m+1} u}{\partial x^{2m+1}} (t,x), \ \ \ t\ge0, \ x\in \mathbb{R},\\
u(0,x) = \delta(x), \ \ \forall\ x,\\
\lim_{x\rightarrow \pm \infty} u(t,x) = 0, \ \ \forall\ t.
\end{cases}
\end{equation}
In particular, the solution to problem (\ref{problemaDifferenzialeIperAiryIntere}) is 
\begin{equation}\label{soluzioneProblemDifferenzialeParzialeIperAiryIntere}
u_{2m+1}(t, x) = \frac{1}{\sqrt[2m+1]{(2m+1)t}} A_{2m+1}\Bigl(\frac{x}{\sqrt[2m+1]{(2m+1) t}}\Bigr).
\end{equation}

Marchione and Orsingher \cite{MO2023} then extended the above idea proving that real valued hyper-Airy function (\ref{soluzioneProblemDifferenzialeParzialeIperAiryReale}) satisfies the fractional differential problem (\ref{problemaParzialeCauchyAiryRealeIntro}) and then they related it to the theory of pseudo-processes. In particular, function (\ref{soluzioneProblemDifferenzialeParzialeIperAiryReale}) is interpreted as the pseudo-density, $u_\alpha$, of the position of a fractional pseudo-process $X_\alpha = \big\{X_{\alpha}(t)\big\}_{t\ge0}$ of order $\alpha>1$, having $x$-Fourier transform of the density equal to
\begin{equation}\label{trasformataFourierPseudoProcesso}
\int_{-\infty}^\infty e^{i\gamma x} u_\alpha(t,x) \dif x = e^{-i\,\text{sgn}(\gamma) |\gamma|^\alpha t},\ \ \ t\ge0,\ \gamma \in\mathbb{R}. 
\end{equation}
We refer to \cite{CO2024b, H1978, K1960, L2003, L2007, MO2023} (and references therein) for details on the construction of the pseudo-processes (signed measure); also see \cite{ODo2012} for further results on the probabilistic representation of the solutions of higher-order heat-type equations.

Our next statement generalizes the results above.

\begin{theorem}\label{proposizioneConvoluzione}
Let $N\in\mathbb{N}$, non-even $\alpha_1,\dots,\alpha_N>1$, 
\begin{equation}\label{coefficienteBetaProblemaDifferenzialeParzialeTotale}
\beta_k = 
\begin{cases}
\begin{array}{l l}
(-1)^n, & \alpha_k = 2m+1,\,n\in\mathbb{N}\\
1, & \alpha_k \not\in\mathbb{N},
\end{array}
\end{cases}
\end{equation}
and real $ a_1,\dots,a_N$. 
The fractional differential problem 
\begin{equation}\label{problemaDifferenzialeSommaOperatoriRiesz}
\begin{cases}
\frac{\partial u}{\partial t} (t,x)= \sum_{k=1}^N a_k D_x^{\alpha_k} u(t,x),\ \ \ t\ge0,\ x\in\mathbb{R}, \\
u(0, x) = \delta(x),\ \ \forall\ x,\\
\lim_{x\rightarrow\pm\infty}u(t,x) = 0,\ \ \forall\ t,
\end{cases}
\end{equation}
has solution
\begin{equation}\label{soluzioneProblemaDerviateParzialiGenerale}
u(t,x) = \frac{1}{\pi}\int_0^\infty \cos\Bigl(xw + t \sum_{k=1}^N \beta_k a_k w^{\alpha_k} \Bigr) \dif w = \Conv_{k=1}^N \ \ \frac{1}{\sqrt[\alpha_k\ ]{\beta_k a_k\alpha_k t}} A_{\alpha_k}\Biggl(\frac{x}{\sqrt[\alpha_k\ ]{\beta_k a_k\alpha_k t}}\Biggr),
\end{equation}
where $A_\alpha$ denotes the hyper-Airy function of order $\alpha>1$ defined in (\ref{funzioneAiryReale}).
\end{theorem}

\begin{proof}
By keeping in mind definition (\ref{trasformataFourierOperatoreRiesz}), we apply the $x$-Fourier transform to problem (\ref{problemaDifferenzialeSommaOperatoriRiesz}),
\begin{equation}\label{trasformataProblemaDifferenzialeSommaOperatoriRiesz}
\begin{cases}
\frac{\partial \mathcal{F}u}{\partial t}(t,\gamma) = -i\,\text{sgn}(\gamma)\sum_{k=1}^N a_k \beta_k|\gamma|^{\alpha_k}\, \mathcal{F}u(t,\gamma),\ \ \ t\ge0,\ x\in\mathbb{R}, \\
\mathcal{F}u(0, \gamma) = 1,\ \ \forall\ \gamma.
\end{cases}
\end{equation}
Now, the solution of (\ref{trasformataProblemaDifferenzialeSommaOperatoriRiesz}) reads
\begin{equation}\label{trasformataFourierSoluzioneProblemaDifferenzialeGenerale}
 \mathcal{F} u(t,\gamma) = e^{-it\,\text{sgn}(\gamma)\sum_{k=1}^N a_k\beta_k|\gamma|^{\alpha_k} }, \ \ \ t\ge0,\ \gamma\in\mathbb{R},
\end{equation}
and by applying the inverse-Fourier transform we achieve
\begin{align*}
 u(t,x) &= \frac{1}{2\pi}\int_{-\infty}^\infty e^{-i\gamma x} e^{-it\,\text{sgn}(\gamma)\sum_{k=1}^N a_k\beta_k|\gamma|^{\alpha_k}} \dif \gamma\\
&= \frac{1}{\pi} \int_{0}^\infty \cos\Bigl(\gamma x +it\,\text{sgn}(\gamma)\sum_{k=1}^N a_k\beta_k|\gamma|^{\alpha_k}\Bigr) \dif \gamma
\end{align*}
which coincides with the second term in (\ref{soluzioneProblemaDerviateParzialiGenerale}).

The last equality of (\ref{soluzioneProblemaDerviateParzialiGenerale}) derives by considering that the $x$-Fourier transform (\ref{trasformataFourierSoluzioneProblemaDifferenzialeGenerale}) can be written as
$$ \mathcal{F} u(t,\gamma) = \prod_{k=1}^N e^{-it\,\text{sgn}(\gamma) a_k\beta_k|\gamma|^{\alpha_k} } = \prod_{k=1}^N \mathcal{F}u_k(t,\gamma), $$ 
where each $u_k$ is the solution to problem (\ref{problemaDifferenzialeSommaOperatoriRiesz}) with $N=1$ and fractional order $\alpha_k$.
\end{proof}

We note that by assuming $a_k = 1/\beta_k =\beta_k$ formula (\ref{soluzioneProblemaDerviateParzialiGenerale}) turns into
\begin{equation}\label{iperAiryCasoSommaDiPseudoProcessi}
 \frac{1}{\pi}\int_0^\infty \cos\Biggl(xw + t \sum_{k=1}^N w^{\alpha_k} \Biggr) \dif w = \Conv_{k=1}^N \ \ \frac{1}{\sqrt[\alpha_k]{\alpha_k t}} A_{\alpha_k}\Bigl(\frac{x}{\sqrt[\alpha_k\ ]{\alpha_kt}}\Bigr). 
\end{equation}
Finally, for $t = 1$  and $\alpha_k = 2m_k +1$ with $n_1,\dots, n_N\in\mathbb{N}$ reads
$$  \frac{1}{\pi}\int_0^\infty \cos\Biggl(xw + \sum_{k=1}^N \frac{w^{2m_k+1}}{2m_k+1} \Biggr) \dif w = \Conv_{k=1}^N \ \ A_{2m_k+1}(x). $$
where the left-hand side represents a particular case of (\ref{generalizzazioneFunzioneAiry}) and the right-hand side satisfies equation (\ref{equazioneConvoluzioneGenerale}) with $c_{n_j} = 1\ \forall\ j$ (in fact, as already explained, the odd-order hyper-Airy function satisfies equation \ref{equazioneDifferenzialeGenerale}, see also (\ref{iperAirySoluzioneEquazioneGenerale})).
\\

As previously said, function (\ref{soluzioneProblemDifferenzialeParzialeIperAiryReale}) represents the pseudo-density of a pseudo-process $\big\{X_{\alpha}(t)\big\}_{t\ge0}$ of the type given in (\ref{trasformataFourierPseudoProcesso}). Therefore, the function (\ref{iperAiryCasoSommaDiPseudoProcessi}), denoting with $\alpha =(\alpha_1,\dots,\alpha_N)$, can be interpreted as the pseudo-density, $u_\alpha$, of the sum of independent pseudo-processes (in the sense that the $x$-Fourier transform of the density coincides with the product of the $x$-Fourier transforms of the densities) of different orders $Y_{N,\alpha} = \Big\{Y_{N,\alpha}(t) = \sum_{k=1}^N X_{\alpha_k}(t)\Big\}_{t\ge0}$, . The corresponding $x$-transform reads
$$ \int_{-\infty}^\infty e^{i\gamma x} u_\alpha(t,x) \dif x = e^{-i\,\text{sgn}(\gamma) t\sum_{k=1}^N |\gamma|^{\alpha_k}},\ \ \ t\ge0,\ \gamma \in\mathbb{R}. $$

\begin{remark}[Even order heat-type equations]\label{remarkCasoParzialePari}
By means of the arguments used in theorem \ref{proposizioneConvoluzione} and by keeping in mind the definition of the Riesz fractional operator (\ref{trasformataFourierOperatoreRiesz}), we can prove that in the case of $\alpha_k = 2m_k,\ k = 1,\dots, N$, assuming $(-1)^{m_k}a_k<0\ \forall\ k$, function 
\begin{equation}\label{soluzioneCasoDifferenzialePari}
u(t,x) = \frac{1}{\pi} \int_0^\infty \cos(\gamma x) e^{t\sum_{k=1}^N (-1)^{m_k} a_k \gamma^{2m_k}}\dif \gamma,\ \ \ t\ge0,\ x\in\mathbb{R},
\end{equation}
satisfies the problem (\ref{problemaDifferenzialeSommaOperatoriRiesz}). Along the proof, following the steps of the proof of Theorem \ref{proposizioneConvoluzione}, one also derive that the $x$-Fourier transform of the function (\ref{soluzioneCasoDifferenzialePari}) is 
\begin{equation}\label{trasformataFourierCasoDifferenzialePari}
\mathcal{F}u(t,\gamma) = e^{t\sum_{k=1}^N (-1)^{m_k} a_k \gamma^{2m_k}},\ \ \ t\ge0,\ \gamma\in\mathbb{R}.
\end{equation}

Note that if $a_k = (-1)^{m_k+1}\ \forall\ k,$ and $t=1$, then (\ref{soluzioneCasoDifferenzialePari}) reads 
\begin{equation}\label{soluzioneCasoDifferenzialePariSpeciale}
u(t,x) = \frac{1}{\pi} \int_0^\infty \cos(\gamma x) e^{-t\sum_{k=1}^N \gamma^{2m_k}}\dif \gamma =  \Conv_{k=1}^N \ g_{2m_k}(x), 
\end{equation}
where $g_{2m_k}$ satisfies equation (\ref{equazioneDifferenzialeGenerale}) with $n=2m_k$ and $c=(-1)^{m_k+1}$. Indeed, for $N=1$ we have that function (\ref{soluzioneCasoDifferenzialePariSpeciale}) satisfies the corresponding problem (\ref{problemaDifferenzialeSommaOperatoriRiesz}) (with $N=1$ and suitable $a_{m_1}$) and it coincides with function (\ref{caso4mDifferenza})  if $m_1$ is even and with function (\ref{caso4m+2Differenza}) if $m_1$ is odd. 
\hfill$\diamond$
\end{remark}

Form Theorem \ref{proposizioneConvoluzione} and Remark \ref{remarkCasoParzialePari} follows the next corollary.

\begin{corollary}
Let $N\in\mathbb{N}$ and real $\alpha_1,\dots,\alpha_N>1$. Let $\beta_k $ as defined in (\ref{coefficienteBetaProblemaDifferenzialeParzialeTotale}) and $I_e$ be the set of the indexes of the even $\alpha_k$. Also assume that $(-1)^{\alpha_{\bar {k}/2}} a_{\bar k}$ with $\bar k = \operatorname*{argmax}_{k\in I_e}\alpha_k$. Then, the solution to the differential problem (\ref{problemaDifferenzialeSommaOperatoriRiesz}) reads
\begin{equation}\label{soluzioneCasoGeneraleDifferenziale}
u(t,x) = \frac{1}{\pi} \int_0^\infty  e^{t\sum_{k\in I_e} (-1)^{\alpha_k/2} a_k \gamma^{\alpha_k}} \cos \Biggl( \gamma x + t\sum_{\substack{k=1\\ k\not\in I_e}}^N a_k\beta_k \gamma^{\alpha_k} \Biggr)\dif \gamma,\ \ \ t\ge0,\ x\in\mathbb{R}.
\end{equation}
\end{corollary}
Note that the condition on the negativity of the coefficient of the higher even exponent is sufficient for the convergence of the solution; requiring the negativity for all $k \in I_e$ would be needed in the case of some convolution representation.

Finally, we can also state the following result connecting the integer higher-order heat equations and the higher-order Airy equations (\ref{equazioneDifferenzialeGenerale}).

\begin{proposition}
Let $n\ge2$ and real $a\not =0$. Let $u:[0,\infty)\times\mathbb{R}\longrightarrow \mathbb{R}$ be the solution to the heat-type differential problem 
\begin{equation}\label{soluzioneProblemaCaloreLegameAiry}
\begin{cases}
\frac{\partial u}{\partial t} (t,x)=  a \frac{\partial^n u}{\partial x^n} u(t,x),\ \ \ t\ge0,\ x\in\mathbb{R}, \\
u(0, x) = \delta(x),\ \ \forall\ x,\\
\lim_{x\rightarrow\pm\infty}u(t,x) = 0,\ \ \forall\ t,
\end{cases}
\end{equation}
with the condition $(-1)^{n/2}a <0$ if $n$ even. Then, for fixed $t\ge0$, the single variable function $u(t,\cdot)$ satisfies the Airy-type equation (\ref{equazioneDifferenzialeGenerale}) with $c = \frac{1}{nta}$.
\end{proposition}

\begin{proof}
We prove the statement for $n=2m+1$ with $m$ even. From (\ref{soluzioneProblemaDerviateParzialiGenerale}) and (\ref{coefficienteBetaProblemaDifferenzialeParzialeTotale}) we can readily write the solution $u$ to problem (\ref{soluzioneProblemaCaloreLegameAiry}), $t\ge0$ and $x\in\mathbb{R}$,
\begin{equation} 
u(t,x) = \frac{1}{\pi}\int_0^\infty \cos\Bigl(xw + t a w^{n} \Bigr) \dif w.
\end{equation}
Now, fix $t>0$, we consider the solution to (\ref{equazioneDifferenzialeGenerale}), with $c = \frac{1}{nta}$, given in (\ref{iperAirySoluzioneEquazioneGenerale}), that is, for $x\in\mathbb{R}$,
\begin{align}
y(x) &= \int_0^\infty \cos \Biggl(\frac{wx}{nta} + \biggl(\frac{1}{nta}\biggr)^{n-1} \frac{w^n}{n} \Biggr) \dif w  \nonumber\\
& = nt|a|\int_0^\infty \cos \Biggl(\frac{\gamma x}{\text{sgn}(a)} + t|a|\gamma^n \Biggr) \dif \gamma \label{passaggioCambioVariabileLegameCaloreAiry}\\ 
& =  nt|a|\pi u(t,x), \nonumber
\end{align}
where in step (\ref{passaggioCambioVariabileLegameCaloreAiry}) we used the change of variable $\frac{w}{nt|a|} =\gamma$ and in the last step we used that the cosine is an even function. The proof ends by observing that the coefficient $nt|a|\pi$ is constant in $x$ and that equation (\ref{equazioneDifferenzialeGenerale}) is linear and homogeneous.

The case of $n = 2m+1$ with $m $ odd follows in a similar way by suitably adapting (\ref{soluzioneProblemaDerviateParzialiGenerale}) and (\ref{iperAirySoluzioneEquazioneGeneraleCasoDispari}). 
\\The case of $n=2m$ can be studied by using (\ref{soluzioneCasoDifferenzialePariSpeciale}) and by keeping in mind function (\ref{caso4mDifferenza}) if $m$ is even and function (\ref{caso4m+2Differenza}) if $m$ is odd. It is very important to note that in these cases the condition $0>(-1)^{n/2}a = (-1)^m a$ reflects the need of $c<0$ in (\ref{caso4mDifferenza}) and $c>0$ in (\ref{caso4m+2Differenza}).
\end{proof}

\begin{remark}[Time-fractional and Riesz-Feller operators]

It is worthwhile considering the differential problem (\ref{problemaDifferenzialeSommaOperatoriRiesz}) involving other fractional operators. We now consider the Dzherbashyan-Caputo fractional derivative, defined as, for $n\in \mathbb{N}_0$ and $\nu>-1$,
\begin{equation*}\label{derivataCaputo}
\frac{\dif^\nu}{\dif t^\nu} f(t) = \begin{cases}
\begin{array}{l l}
\displaystyle\frac{1}{\Gamma(n-\nu)}\int_0^t (t-s)^{n-\nu-1}\frac{\dif^n}{\dif s^n}f(s)\dif s, & \ \text{if}\ n-1<\nu<n,\\[9pt]
\displaystyle\frac{\dif^n}{\dif t^n} f(t), &\ \text{if} \ \nu = n,
\end{array}\end{cases}
\end{equation*}
(we refer to \cite{M2010, P1999} for some elements on fractional calculus theory) and the Riesz-Feller (space) fractional operator, defined through its Fourier transform, that is
$$ \mathcal{F} D^{\alpha\theta,\theta}_x f (\gamma)= -|\gamma|^{\alpha\theta} e^{i\,\text{sgn}(\gamma) \frac{\pi\theta}{2} } \mathcal{F}f(\gamma),\ \ \ \gamma\in\mathbb{R}. $$

Now, let $\nu>0$ and $\theta \in(0,1)$. Also consider $\alpha = (\alpha_1,\dots,\alpha_N)$ with $\alpha_k\theta>1\ \forall\ k$. In light of Theorem 3.1 of \cite{CO2024b} and result (42) of \cite{MO2023}, the solution to the time-space fractional differential problem, with $\beta_k$ given in (\ref{coefficienteBetaProblemaDifferenzialeParzialeTotale}),
\begin{equation}\label{problemaDifferenzialeSommaOperatoriRiesz}
\begin{cases}
\frac{\partial^\nu u}{\partial t^\nu} = \sum_{k=1}^N \beta_k D^{\alpha_k\theta,\theta}_x u(t,x),\ \ \ t\ge0,\ x\in\mathbb{R}, \\
u(0, x) = \delta(x),\ \ \forall\ x,\\
\lim_{x\rightarrow\pm\infty}u(t,x) = 0,\ \ \forall\ t,
\end{cases}
\end{equation}
is the pseudo-density of the position of a sum of independent time-changed pseudo-processes 
$$Y_{N, \alpha}\Bigl( S_\theta\bigl(L_\nu(t)\bigr) \Bigr) = \sum_{k=1}^N X_{\alpha_k}\Bigl( S_\theta\bigl(L_\nu(t)\bigr) \Bigr), $$
where $X_{\alpha_k}$ are independent pseudo-processes of the type in (\ref{trasformataFourierPseudoProcesso}), $S_\theta$ is an independent stable subordinator of order $\theta$ and $L_\nu$ is an independent inverse of a pseudo-subordinator of order $\nu$. We refer to \cite{CO2024b, Z1986} for further details on subordinators, pseudo-subordinators and their inverses and to \cite{CO2024} for further results on time-changed stochastic processes relate to higher-dimensional fractional equations.

We point out that one can easily study the case of a slightly more articulated fractional time operator like showed in Example 4.2 of \cite{CO2024b}.
\hfill$\diamond$
\end{remark}

\begin{remark}[Moments]
It is interesting to study the moments of pseudo-processes. Here we show a method to derive all the moments of the odd pseudo-process $X_{2m+1}$ with $m\in\mathbb{N}$, having pseudo-density (\ref{soluzioneProblemDifferenzialeParzialeIperAiryIntere}). 
By keeping in mind formula (\ref{trasformataFourierSoluzioneProblemaDifferenzialeGenerale}) and by setting $N=1,\ a_1 = a\beta_1 = a(-1)^m,$ with $a\not =0$ (set $a=1$ for the classical density of the pseudo-process) since $\alpha_1 = 2n+1$, we have the $x$-Fourier transform of (\ref{soluzioneProblemDifferenzialeParzialeIperAiryIntere}),
\begin{align}
\mathcal{F}u_{2m+1}(t, \gamma) =  \int_{-\infty}^\infty e^{i\gamma x}u(t,x)\dif x  = e^{-iat \gamma^{2m+1} }
\end{align}
which yields
\begin{align}
\sum_{h=0}^\infty   \frac{(i\gamma)^h}{h!} \int_{-\infty}^\infty x^h u_{2m+1}(t,x)\dif x  &= \sum_{k=0}^\infty \frac{\bigl(-iat \gamma^{2m+1}\bigr)^k }{k!}.\nonumber 
\end{align}
For the identity principle of polynomials it must hold that for each $h\in\mathbb{N}_0$,
\begin{equation}
\frac{(i\gamma)^h}{h!} \int_{-\infty}^\infty x^h u(t,x)\dif x =
\begin{cases}\begin{array}{l l}
\frac{(-iat)^k \gamma^{(2m+1)k} }{k!},  & \ \exists\ k\in\mathbb{N}_0 \text{ s.t. } h = (2m+1)k,\\
0, & \ h\not=(2m+1)k, \ \forall\ k\in\mathbb{N}_0.
\end{array}\end{cases}
\end{equation}
Therefore, the non-zero moments are
\begin{equation}\label{momentoDispari}
 \int_{-\infty}^\infty x^{(2m+1)k} u_{2m+1}(t,x)\dif x = (-1)^{(m+1)k}\, (at)^k\frac{\bigl((2m+1)k\bigr)! }{k!}, \ \ \ k\in\mathbb{N}_0.
\end{equation}
Note that according to $m$ and $a$ the moments may change sign cyclically.

In light of Theorem \ref{proposizioneConvoluzione} and the above results, we can express the moments of the functions of type (\ref{iperAiryCasoSommaDiPseudoProcessi}) with odd $\alpha_1,\dots,\alpha_N$ in terms of those in (\ref{momentoDispari}). Indeed, with 
\begin{align}\label{momentiConvoluzione}
\int_{-\infty}^\infty x^h \frac{1}{\pi}\int_0^\infty \cos\Biggl(xw + t \sum_{k=1}^N w^{\alpha_k} \Biggr) \dif w \dif x &= \sum_{\substack{m_1,\dots,m_N=1\\ m_1+\dots +m_N = h}}^h \binom{h}{m_1,\dots,m_N} \prod_{k=1}^N \int_{-\infty}^\infty x^{m_k} u_{\alpha_k}(t,x)\dif x,
\end{align}
where we also used the fact that the first term of (\ref{momentiConvoluzione}) is the pseudo-density of the sum of independent pseudo-processes.
\\

We point out that in the even case, $\alpha_1=2n$, with $a_1 = (-1)^{m+1}$, non-zero moments are given by
\begin{equation*}
 \int_{-\infty}^\infty x^{2mk} u_{2m}(t,x)\dif x = (-1)^{(m+1)k}\, t^k\frac{\bigl(2mk\bigr)! }{k!}, \ \ \ k\in\mathbb{N},
\end{equation*}
which is equivalent to (\ref{momentoDispari}) and it follows by means of (\ref{trasformataFourierCasoDifferenzialePari}) and the above arguments. Note that for $m$ even the moments change sign cyclically.
\hfill$\diamond$
\end{remark}


\subsection*{\large{Declarations}}

\textbf{Ethical Approval.} This declaration is not applicable.
 \\
\textbf{Competing interests.}  The authors have no competing interests to declare.
\\
\textbf{Authors' contributions.} Both authors equally contributed in the preparation and the writing of the paper.
 \\
\textbf{Funding.} The authors received no funding.
 \\
\textbf{Availability of data and materials.} This declaration is not applicable.







\begin{thebibliography}{00}

\bibitem{A1838}
Airy, G.B. (1838), On the intensity of light in the neighbourhood of a caustic, Trans. Camb. Phil. Soc. 6, 379--402.

\bibitem{A2017}
Ansari, A. (2017) Riesz fractional derivatives of solutions of differentual equation $x^{(4)} + xy = 0$, Studia Scientiarum Mathematicarum Hungarica 54(2), 205--220. 

\bibitem{AA2014}
Ansari, A., Askari, H. (2014), On fractional calculus of $A_{2n+1}(x)$ function, Applied Mathematics and Computation 232, 487--497.

\bibitem{AA2020}
Askari, H., Ansari, A.  (2020), On Mellin transforms of solutions of differential equation $\chi^{(n)}(x)+\gamma_nx\chi(x)=0$, Analysis and Mathematical Physics 10, no.4, 57.

\bibitem{AA2024}
Askari, H., Ansari, A., Stokes phenomenon and representations of M-wright functions in uniform asymptotic expansions with many coalescing saddle points: order $\frac{1}{n}$, Preprint (2024).


\bibitem{BKOS2007}
L. Beghin, L., Kozachenko, Y., Orsingher, E., Sakhno, L. (2007), On the solutions of linear odd-order heat-type equations with random initial conditions, J. Stat. Phys. 127(4), 721--739.

\bibitem{BM2015}
Bonaccorsi, S., Mazzucchi, S. (2015), High order heat-type equations and random walks on the complex plane, Stoch. Process. Appl. 125(2), 797--818.


%



\bibitem{CO2024}
Cinque, F., Orsingher, E. (2024), Analysis of Fractional Cauchy problems with some probabilistic applications, J. Math. Anal. Appl. 536, 128188.

\bibitem{CO2024b}
Cinque, F., Orsingher, E., Higher-order fractional equations and related time-changed pseudo-processes, Preprint (2024) arXiv:2402.13691.

\bibitem{DF1965}
Daletsky, Y.L., Fomin, S.V. (1965), Generalized measures in function spaces, Theory Probab. Appl. 10(2), 304--316.

\bibitem{D2006}
Debbi, L. (2006), Explicit solutions of some fractional partial differential equations via stable subordinators, J. Appl. Math. Dtoch. Anal. 5:093502.




%
%


\bibitem{GHPD2013}
G\'{o}rska, K., Horzela, A., Penson, K.A., Dattoli, G. (2013), The higher-order heat-type equations via signed L\'{e}vy stable and generalized Airy functions, Journal of Physics A: Mathematical and theoretical 46, 425001.






\bibitem{H1978}
Hochberg, K.J. (1978), A signed measure on path space related to wiener measure, Ann. Probab. 6(3), 433--458.

%
%
%
%


\bibitem{K1960}
Krylov, V.Yu. (1960), Some properties of the distribution corresponding to equation $\frac{\partial u}{\partial t} = (-1)^{q+1}\frac{\partial^{2q} u}{\partial x^{2q}}$, Dokl. Akad. Nauk SSSR 132(6), 1254--1257.

\bibitem{L2003}
Lachal, A. (2003), Distributions of sojourn time, maximum and minimum for pseudo-processes governed by higher-order heat-type equations, Elect. J. Probab. 8:1--53.


\bibitem{L2007}
Lachal, A. (2007), First hitting time and place, monopoles and multipoles for pseudo-processes driven by the equation $\partial u/\partial t = \pm \partial^N u/ \partial x^N$, Electr. J. Probab. 12:300-353.

\bibitem{L2012}
Lachal, A. (2012), A survey on the pseudo.process driven by the high-order heat-type equation $\partial /\partial t =\pm \partial^N/\partial x^N$ concerning the hitting and sojourn times, Methodol. Comput. Appl. Probab. 14, 549--566.




\bibitem{M2010}
Mainardi, F., Fractional Calculus and Waves in Linear Viscoelasticity, Imperial College Press, London, 2010.

\bibitem{MO2023}
Marchione, M.M., Orsingher, E. (2023), Stable distributions and pseudo-processes related to fractional Airy functions, Stochastic Analysis and Applications 42(2), 435--450.





%
%

\bibitem{ODo2012}
Orsingher, E., D'Ovidio, M. (2012), Probabilistic representation of fundamental solutions to $\frac{\partial u}{\partial t} = \kappa_m\frac{\partial^m u}{\partial x^m}$, Electronic Communications in Probability 17, 1--12.




\bibitem{OT2014}
Orsingher, E., Toaldo, B. (2014), Pseudoprocesses related to space-fractional higher-order heat-type equations, Stochastic Analysis and Applications 32(4), 619--641.



\bibitem{P1999}
Podlubny, I., Fractional Differential Equations, Mathematics in Sciences and Engineering , Academic Press, San Diego, 1999.



\bibitem{R1995}
Reid, W.H. (1995), Integral representations for products of Airy functions, Zeitschrift f\"{u}r angewandte Mathematik und Physik 46, 159--170.
\bibitem{R1997}
Reid, W.H. (1997), Integral representations for products of Airy functions Part 2: Cubic products, Zeitschrift f\"{u}r angewandte Mathematik und Physik 48, 646--655.
\bibitem{R1997b}
Reid, W.H. (1997), Integral representations for products of Airy functions, Part 3: Quartic products, Zeitschrift f\"{u}r angewandte Mathematik und Physik 48, 656--664.


\bibitem{S1950}
Scorer, R.S. (1950), Numerical evaluation of integrals of the form $I = \int_{x_1}^{x_2} f(x) e^{i\phi(x)} dx$ and the tabulation of the function $Gi(z) = (1/\pi)\int_0^\infty\sin(uz+\frac{1}{3}u^3)du$, Quart. Journ. Mech. and Applied Math., Vol. III, Pt. 1.




\bibitem{VS2010}
Vallée, O., Soares, M., Airy functions and applications to physics, 2nd edition, Imperial College Press, London, 2010.

\bibitem{V2010}
Varlamov, V. (2010), Integrals involving products of Airy functions, their derivatives and Bessel functions, J. Math. Anal. Appl. 370, 687--702.

\bibitem{Z1986}
Zolotarev, V.M. (1986), One-Dimensional Stable Distributions, Translations of Mathematical Monographs, vol.65, American Mathematical Society.

\end{thebibliography}

\footnotesize{

}

\end{document}